\documentclass[10pt]{article}

\usepackage{amsthm,amssymb,amsbsy,amsmath,amsfonts,amssymb,amscd}
\usepackage{subfigure, hyperref}
\usepackage{mypackage}
\usepackage{setspace}
\usepackage[letterpaper]{geometry} 
\usepackage{url}
\theoremstyle{plain}

\theoremstyle{definition}
\newtheorem{defn}[theorem]{Definition}

\theoremstyle{remark}
\hypersetup{ 
backref=true, 
pagebackref=true,
hyperindex=true, 
colorlinks=true, 
breaklinks=true, 
urlcolor= blue, 
linkcolor= blue, 
bookmarks=true, 
bookmarksopen=true, 
bookmarksnumbered=true} 
\date{\today}
\title{Detection of a Moving Rigid Solid in a Perfect Fluid}
\author{Carlos Conca\footnote{C. Conca thanks the MICDB for partial support through Grant ICM P05-001-F, Fondap-Basal-Conicyt, and the French \& Chilean Governments through Ecos-Conicyt Grant C07E05.} \\{\small Center for Mathematical Modelling,}\\
{\small Av. Blanco Encalada 2120 Piso 7,}\\
{\small University of Chile, Santiago, Chile,}\\
 {\small \texttt{cconca@dim.uchile.cl}}
\and Muslim Malik\\
{\small Center for Mathematical Modelling,}\\
{\small Av. Blanco Encalada 2120 Piso 7,}\\
{\small University of Chile, Santiago, Chile,}\\
{\small \texttt{malikiisc@gmail.com}}
\and Alexandre Munnier\footnote{Author supported by ANR CISIFS and ANR GAOS.} \\
{\small Institut Elie Cartan UMR 7502, Nancy-Universit\'e,}\\
{\small CNRS, INRIA, B.P. 239,}\\
{\small F-54506 Vandoeuvre-l\`es-Nancy Cedex, France,}\\
{\small \texttt{alexandre.munnier@iecn.u-nancy.fr}}}
\begin{document}
\maketitle
\begin{abstract} 
In this paper, we consider a moving rigid solid immersed in a potential fluid. The fluid-solid system fills the whole two dimensional space and the fluid is assumed to be at rest at infinity. 
Our aim is to study the inverse problem, initially introduced in \cite{Conca:2008ab}, that consists in recovering the position and the velocity of the solid assuming that the potential function is known at a given time. 

We show that this problem is in general ill-posed by providing counterexamples for which the same potential corresponds to different positions and velocities of a same solid.
However, it is also possible to find solids having a specific shape, like ellipses for instance, for which  the problem of detection admits a unique solution. 

Using complex analysis, we prove that the well-posedness of the inverse problem is equivalent to the solvability of an infinite set of nonlinear equations. 
This result allows us to show that when the solid enjoys some symmetry properties, it can be {\it partially} detected. Besides, for any solid, the velocity can always be recovered when both the potential function and the position are supposed to be known.  

Finally, we prove that by performing continuous measurements of the fluid potential over a time interval, we can always track the position of the solid.
\end{abstract}


\section {\bf Introduction}
\subsection{History} 
Sonars are the most common device used to spot immersed bodies, like submarines or banks of fish. These systems use acoustic waves: active sonars emit acoustic waves (making themselves detectable), while passive sonars only listen (and therefore are only able to detect targets that are noisy enough). To overcome these limitations, it would be interesting to design systems imitating the {\it lateral line systems} of fish, a sense organ they use to detect movement and vibration in the surrounding water.

Most of the published results on inverse problems in Fluid Mechanics concern the detection of fixed immersed obstacles. For example in~\cite{Alvarez:2005aa} the authors prove that a fixed
smooth convex obstacle surrounded by a fluid governed
by the Navier-Stokes equations can be identified via a localized
boundary measurement of the velocity of the fluid and the Cauchy
forces. In~\cite{Doubova:2007aa}, the authors identify a single rigid obstacle immersed in a Navier-Stokes fluid by measuring 
both the gradient of the pressure and the velocity
of the fluid on one part of the boundary.  The distance from a given point to an obstacle is estimated in~\cite{Heck:2007aa} from boundary measurements for a fluid governed by the stationary Stokes equations.

To our knowledge, the only work addressing the detection of moving bodies is \cite{Conca:2008ab}. In this paper, the authors consider a single moving disk in an ideal fluid and prove that the position and velocity of the body can be deduced from one single measurement of the potential along some part of the exterior boundary of the fluid. They obtain linear stability results as well, by using shape
differentiation techniques.

\subsection{Problem settings}
\subsubsection*{Domains, frames, coordinates}
At a given time $t$, we assume that a rigid solid occupies the domain $\mathcal S\subset \mathbf R^2$, while the domain $\mathcal F:=\mathbf R^2\setminus\bar{\mathcal S}$ is filled by a perfect fluid.
Let us assume that $\mathcal S$ is a simply connected compact set. The unitary normal to $\partial\mathcal F$ directed towards the exterior of $\mathcal F$ is denoted by $\mathbf n$.  As being a rigid solid, $\mathcal S$ is the image by a rotation and a translation of a given reference domain $\mathcal S_0$ which will merely be called in the sequel the {\it shape} of the solid. Therefore, at any time, there exist an angle $\theta\in\mathbf R/2\pi$, a rotation matrix $R(\theta)\in {\rm SO(2)}$ of angle $\theta$, a point $\mathbf s:=(s_1,s_2)^T\in\mathbf R^2$ (the center of the rotation) and a vector $\mathbf r:=(r_1,r_2)^T$ of $\mathbf R^2$ such that $\mathcal S=R(\theta)(\mathcal S_0-\mathbf s)+\mathbf r$. We will be concerned with recovering the position of the solid, so it is worth remarking that the triplet $(\theta,\mathbf s,\mathbf r)$ is not unique. Two triplets $(\theta_j,\mathbf s_j,\mathbf r_j)\in \mathbf R/2\pi\times\mathbf R^2\times \mathbf R^2$ $(j=1,2)$ give the same position for any $\mathcal S_0$ if and only if $R(\theta_1)=R(\theta_2)=R$ and $R(\mathbf s_1-\mathbf s_2)=\mathbf r_1-\mathbf r_2$. These equalities define an equivalence relation in $\mathbf R/2\pi\times\mathbf R^2\times \mathbf R^2$. However, we want also to take into account the possible symmetries of the solid. So, given $\mathcal S_0$, we say that two triplets $(\theta_j,\mathbf s_j,\mathbf r_j)$ are equivalent when $R(\theta_1)(\mathcal S_0-\mathbf s_1)+\mathbf r_1=R(\theta_2)(\mathcal S_0-\mathbf s_2)+\mathbf r_2$. We denote by $\mathcal P$ the set of all of the equivalence class $\mathbf p$. We will make no difference in the notation between $\mathbf p$ and any element $(\theta,\mathbf s,\mathbf r)$ belonging to this class. In particular, we will write in short that for any $x\in\mathbf R^2$, $\mathbf px=R(\theta)(x-\mathbf s)+\mathbf r$. In the sequel, $\mathbf p$ will be merely referred to as {\it position} of the solid.

Later on, we will use tools of complex analysis, so rather than $\mathbf R^2$, we will sometimes identify the plane with the complex field $\mathbf C$. For any complex number $z:=z_1+iz_2$ ($i^2=-1$, $z_1,z_2\in\mathbf R$), we will denote $\bar z:=z_1-iz_2$ the conjugate of $z$ and $D$ will stand for the unitary disk of $\mathbf C$. 
\subsubsection*{Sequences of complex numbers}
For any sequence of complex numbers $c:=(c_k)_{k\in\mathbf Z}$, we can define $\bar c:=(\bar c_k)_{k\in\mathbf Z}$ and $\check c:=(\bar c_{-k})_{k\in\mathbf Z}$. For any two sequences $a:=(a_k)_{k\in\mathbf Z}$ and $b:=(b_k)_{k\in\mathbf Z}$, we recall the definition of the convolution product: $a\ast b:=(\sum_{j\in\mathbf Z} a_{k-j}b_j)_{k\in\mathbf Z}$. The convolution product can be iterated $n$ times ($n$ an integer) to obtain $a^n:=a\ast a\ast\ldots\ast a$. 

\subsubsection*{Rigid velocity}
The solid is moving. We denote by $\mathbf v(x):=(v_1(x),v_2(x))^T\in\mathbf R^2$ the rigid Eulerian velocity field defined for all $x\in\mathcal S$. This notation turns out to be $v(z):=v_1(z)+iv_2(z)$ in complex notation. It is well known in Solid Mechanics that  $\mathbf v$ can be decomposed into the sum of an instantaneous rotational velocity field and a translational velocity field. For any $x:=(x_1,x_2)^T\in\mathbf R^2$, we introduce the notation $x^\perp:=(-x_2,x_1)^T$ and we have $\mathbf v(x)=\omega(x-\mathbf s)^\perp+\mathbf w$, where $\mathbf s\in\mathbf R^2$ is the center of the instantaneous rotation, $\omega\in\mathbf R$ is the angular velocity and $\mathbf w:=(w_1,w_2)^T\in\mathbf R^2$ the translational velocity. 
Since we wish to recover these data, it is worth observing that the triplet $(\omega,\mathbf s,\mathbf w)\in\mathbf R\times\mathbf R^2\times\mathbf R^2$ is not unique. Both triplets $(\omega_j,\mathbf s_j,\mathbf w_j)$ ($j=1,2$) give the same rigid velocity field $\mathbf v$ if and only if $\omega_1=\omega_2=\omega$ and $\omega(\mathbf s_1^\perp-\mathbf s_2^\perp)=\mathbf w_1-\mathbf w_2$ (in particular, we can always choose for $\mathbf s$ any point of $\mathbf R^2$). This is an equivalence relation and the velocity $\mathbf v$ can be seen as an equivalence class. We denote $\mathcal V$ as the set of all of the equivalence class and we will not differentiate, in what follows, between the vector field, the class of equivalence and any element of this class. All of them will be denoted by $\mathbf v$.
\begin{defn}[Configurations]
For any given shape $\mathcal S_0$, we define a {\it configuration} as any position-velocity pair $(\mathbf p,\mathbf v)\in\mathcal P\times\mathcal V$. 
\end{defn} 
\subsubsection*{Fluid dynamics}
The dynamics of the fluid is described by means of its Eulerian velocity field $\mathbf u(x):=(u_1(x),u_2(x))^T$ defined for all $x\in\mathcal F$. Since the fluid is assumed to be perfect (i.e. incompressible and inviscid) and the flow irrotational, there exists a potential function $\varphi$, harmonic in $\mathcal F$, such that $\mathbf u(x)=\nabla\varphi(x)$ ($x\in\mathcal F$). The fluid is assumed to be at rest at infinity so we impose the asymptotic behavior $|\nabla \varphi(x)|\to 0$ as $|x|\to+\infty$. The classical {\it slip} boundary condition for inviscid fluid reads as $\mathbf u\cdot \mathbf n=\mathbf v\cdot \mathbf n$ on $\partial \mathcal S$ and yields a Neumann boundary condition for $\varphi$, namely $\partial_n\varphi =\mathbf v\cdot \mathbf n$ on $\partial\mathcal S$. Although the domain $\mathcal F$ is not simply connected, we can still consider $\psi$, the harmonic conjugate function to $\varphi$, because $\int_{\partial \mathcal S}\partial_n\varphi{\rm d}\sigma =0$. The functions $\varphi$ and $\psi$ satisfy the relation $\nabla\psi=(\nabla\varphi)^\perp$ in $\mathcal F$. In Fluid Mechanics, $\psi$ is called the stream function and the complex function $\xi=\varphi+i\psi$ is the holomorphic complex potential. As usual, we define $u:=u_1+iu_2=\bar\xi'$ as the complex fluid velocity. Observe that the complex potential, as being solution of a boundary value problem, depends on the domain $\mathcal F$ and the velocity $\mathbf v$ only. With the notation introduced earlier, we deduce that $\xi$ depends only on the shape $\mathcal S_0$ and the configuration $(\mathbf p,\mathbf v)\in\mathcal P\times\mathcal V$.

The complex potential is defined up to an additive constant which can be chosen such that $|\xi(z)|\to 0$ as $|z|\to+\infty$. For any $\nu\in\mathbf C$, the complex potential can be expanded in the form of a Laurent series:
\begin{equation}
\label{a:priori}
\xi(z):=\sum_{j\geq 1}\frac{\lambda_j(\nu)}{(z-\nu)^j},\quad |z-\nu|>R(\nu),
\end{equation}
where $\lambda_j(\nu)$ ($j\geq 1$) are complex numbers and $R(\nu):=\limsup_{j\to+\infty}|\lambda_j(\nu)|^{1/j}$. The series is uniformly convergent on the set $\{z\in\mathbf C\,:\, |z-\nu|>R(\nu)\}$.
\subsubsection*{Measurements}
We measure the complex velocity $u$ of the fluid in some open subset of $\mathcal F$. The Analytic Continuation theorem tells us that we can 
deduce the value of  $\xi'$ everywhere in the connected open set $\mathcal F$ and then also the value of $\xi$, up to an additive constant. In particular, we will assume that for all $\nu\in\mathbf C$, 
we can always evaluate all of the terms of the complex sequence $(\lambda(\nu))_{j\geq 1}$ arising in the expression \eqref{a:priori}.
\subsection{Main results}
\begin{defn}[Detectability]
A solid of shape $\mathcal S_0$ is said to be detectable if, for any configuration $(\mathbf p,\mathbf v)\in\mathcal P\times\mathcal V$, the knowledge of the potential holomorphic function $\xi$ suffices for recovering the pair  $(\mathbf p,\mathbf v)$.
\end{defn}
Observe that this definition makes the property of being detectable independent of the configuration:  detectability is a purely geometric property of the solid. Our first result is that not all the solids are detectable:
\begin{theorem}
\label{theo:1}
For any integer $n\geq 2$, there exists a holomorphic function $\xi$, a shape $\mathcal S_0$, and $n$ configurations $(\mathbf p_j,\mathbf v _j)\in\mathcal P\times\mathcal V$, $j=1,\ldots,n$  satisfying $\mathbf p_j\neq \mathbf p_k$ if $j\neq k$ such that $\xi$ is the potential of the fluid corresponding to the solid of shape $\mathcal S_0$ with any of the configurations $(\mathbf p_j,\mathbf v_j)$, $j=1,\ldots,n$.
\end{theorem}
In other words, for any integer $n$, there exists at least one solid that can occupy $n$ different positions with $n$ different velocities and for which the fluid potential is the same.
This theorem shows that the result obtained in \cite{Conca:2008ab} for a disk can not be generalized to any solid. However, not only the disk is a detectable body:
\begin{prop}
\label{prop:ellipse}
Any ellipse is a detectable solid.
\end{prop}
Going back to the general case, it is easy to see that the holomorphic potential never admits an analytic continuation over the whole complex plane. Furthermore, for any analytic continuation of the potential inside the solid, we will prove that the location of the singularities provides clues allowing one in many cases to determine the position of the solid. This discussion is carried out in Subection~\ref{subsect:singul}.

According to Theorem~\ref{theo:1}, the problem of detection is ill-posed in the general case.  However, we claim that when the solid enjoys some symmetry properties, it can be {\it partially detected} (i.e. some but not all of the parameters among $\mathbf r,\alpha, \mathbf w,\omega$ can be deduced from the potential). The following proposition illustrates this idea:
\begin{prop}
If the shape of the solid is invariant under a rotation of angle $\pi/2$ then $\mathbf r$, $\mathbf w$ and $|\omega|$ can be deduced from the potential function.
\end{prop}
We refer to Propositions~\ref{prop:first:detect} and \ref{prop:first:detect:2} for a more precise statement of this result.

In the general case, we can also try to determine less parameters with more information. For instance, we can prove:
\begin{prop}
For any solid with configuration $(\mathbf p,\mathbf v)\in\mathcal P\times\mathcal V$,
the knowledge of both the potential function and the position $\mathbf p$ suffices for recovering $\mathbf v$. 
\end{prop}
Finally, we can also measure the potential function, not only at a given instant, but over a time interval. In this case, we obtain:
\begin{theorem}[Tracking]
\label{theo:tracking}
For any solid $\mathcal S_0$, if we know its position at the time $t=0$ and we perform continuous measurements of the complex potential over the time interval $[0,T]$ for some $T>0$ then we can deduce the configuration of the solid at any time $t\in[0,T]$.
\end{theorem}
\subsection{Outline of the paper}
In the next section, we provide examples of non-detectable solids and prove Theorem~\ref{theo:1}. In Section~\ref{complex:poten}, we derive the expression of the complex potential. In Section~\ref{sec:stealth}, we determine all the {\it stealth} solids, i.e. all the solids that can move in the fluid without disturbing it. The detection of a moving ellipse is discussed in Section~\ref{detection:ellipse}. Section~\ref{singu} is split into three parts: the first one is dedicated to the study of the singularities of the potential function and the second one to its asymptotic expansion and how these results can be used for the detection problem we are dealing with. The third subsection deals with an example of detection. In Section~\ref{tracking} we give the proof of Theorem~\ref{theo:tracking} and at last in Section~\ref{open}, we indicate some remaining open problems.

\section{Examples of Non-detectable Solids}
\label{sec:contre_exx}
This Section is mostly devoted to the proof of Theorem~\ref{theo:1}. 
\subsubsection*{Expression of the stream function}
Let a shape $\mathcal S_0$ and a configuration $(\mathbf p,\mathbf v)$ be given with $\mathbf v=\omega(x-\mathbf s)^\perp+\mathbf w$ (for some real number $\omega$ and some vector $\mathbf s$) and remember that $\mathcal S=\mathbf p(\mathcal S_0)$. Then, let us introduce $\gamma:[0,\ell[\mapsto \gamma(s)=(\gamma_1(s),\gamma_2(s))^T\in\mathbf R^2$ a parameterization of $\partial\mathcal S$ satisfying  $|\gamma'(s)|=1$ for all $s\in[0,\ell[$ ($\ell>0$). We assume that $\partial\mathcal S$ is described positively (counterclockwise parameterization), we denote $\boldsymbol\tau=\gamma'$ (the unitary tangent vector to $\partial\mathcal S$) and we get $\mathbf n=\boldsymbol\tau^\perp$. We deduce that $\partial_n\varphi=-\partial_\tau\psi$ and hence that $\partial_\tau\psi(\gamma)=-w_1\gamma_2'+w_2\gamma_1'+\omega\gamma'\cdot(\gamma-\mathbf s)$. We can integrate along $\partial\mathcal S$ to obtain $\psi(\gamma)=-w_1\gamma_2+w_2\gamma_1+(\omega/2)|\gamma-\mathbf s|^2+C$ on $\partial\mathcal S$, where $C$ is real constant. This Dirichlet boundary condition for the stream function reads also: $\psi(x)=-w_1x_2+w_2x_1+(\omega/2)|x-\mathbf s|^2+C$ on $\partial\mathcal S$. In this form, the boundary of the solid turns out to be a level set of the function $g(x):=(\omega/2)|x-\mathbf s|^2-w_1x_2+w_2x_1-\psi(x)$, an observation we will now take advantage of.

\subsubsection*{Proof of Theorem~\ref{theo:1}}
Pick some integer $n\geq 2$ and consider the harmonic function whose expression in polar coordinates is $\psi(r,\theta):=\cos(n\theta)r^{-n}$. 
Since, in Cartesian coordinates, $|(\partial\psi/\partial x_j)(x)|\leq |\nabla\psi(x)|=n|x|^{-n-1}$ ($j=1,2$), we deduce that for any $\omega>0$ and $\mathbf s\in\mathbf R^2$, there exists $\delta>0$ such that $(\partial\psi/\partial x_1)(x)-\omega(x_1-s_1)$ and $(\partial\psi/\partial x_2)(x)-\omega(x_2-s_2)$ can not be simultaneously null providing $|x-\mathbf s|>\delta$. Applying the Local Inversion Theorem, we deduce that for any $\lambda\in\mathbf R$, the solutions of 
\begin{equation}
\frac{\omega}{2}|x-\mathbf s|^2-\psi(x)-\lambda=0,
\end{equation}
satisfying $|x-\mathbf s|>\delta$ (if any) are locally smooth curves. 

From the estimate $|\psi(x)|\leq|x|^{-n}$ ($x\in\mathbf R^2$), we deduce that for all $\mathbf s:=(s_1,s_2)^T$ and all $\omega>\varepsilon>0$, there exists $\delta'>0$ such that:
$$\frac{\omega-\varepsilon}{2}|x-\mathbf s|^2-\lambda\leq \frac{\omega}{2}|x-\mathbf s|^2-\psi(x)-\lambda\leq \frac{\omega+\varepsilon}{2}|x-\mathbf s|^2-\lambda,$$
for all $\lambda\in\mathbf R$ providing $|x-\mathbf s|\geq \delta'$. If we choose for instance $\lambda>\max(\delta^{\prime},\delta)^2(\omega+\varepsilon)$, there is a zero level set of the function $g(x):=\omega|x-\mathbf s|^2/2-\psi(x)-\lambda$ between the circles $|x-\mathbf s|=\sqrt{2\lambda}/\sqrt{\omega-\varepsilon}$ and $|x-\mathbf s|=\sqrt{2\lambda}/\sqrt{\omega+\varepsilon}$ (because $\sqrt{2\lambda}/\sqrt{\omega+\varepsilon}\geq\sqrt{2}\delta>\delta$).  It remains to choose $\mathbf s$ properly, in order to take advantage of the symmetry of the function $\psi$. Let $\rho$ be any positive number and denote $\mathcal S_1$ the zero level set of $g$ obtained as described above by specifying $\mathbf s_1:=(\rho,0)$. This level set defines the smooth boundary of a solid for which $\psi$ is the stream function (and $\xi(z):=i/z^n$ the holomorphic potential) associated with the velocity $\mathbf v_1:=\omega(x-\mathbf s_1)^\perp$. By choosing next $\mathbf s_k=(\rho\cos(2(k-1)\pi/n),\rho\sin(2(k-1)\pi/n))^T$ for $k=2,\ldots,n$, we obtained $n-1$ copies of $\mathcal S_1$ at $n-1$ different positions with respective velocities $\mathbf v_k:=\omega(x-\mathbf s_k)^\perp$. Some examples of such solids with the associated rigid velocity fields are displayed in Figures~\ref{fig:1}, \ref{fig:2} and \ref{fig:3}.
\begin{figure}[h]     
     \centering
     \begin{tabular}{|cc|}
     \hline
     \subfigure
     [$\omega=0.7$, $\mathbf s_1=(0.7,0)$, $\lambda=1.5$]
     {\includegraphics[width=.35\textwidth]{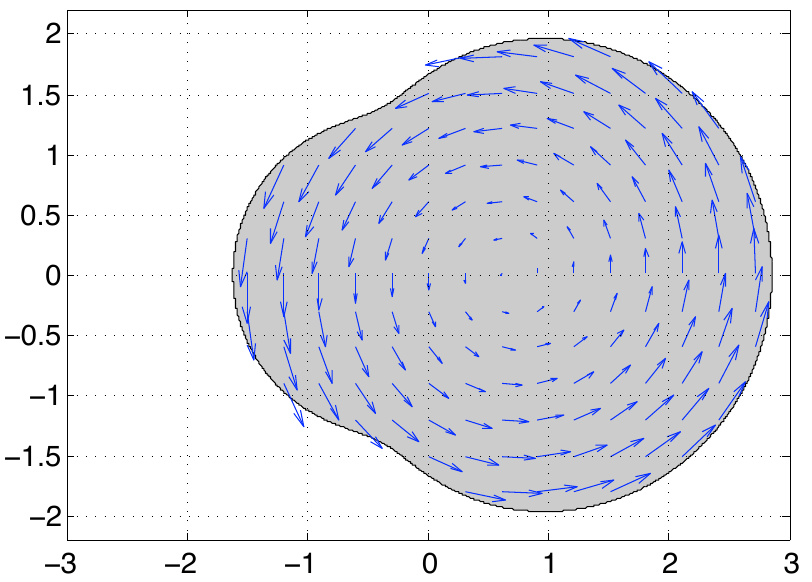}}
     &
     \subfigure
     [$\omega=0.7$, $\mathbf s_2=(-0.7,0)$, $\lambda=1.5$]
     {\includegraphics[width=.35\textwidth]{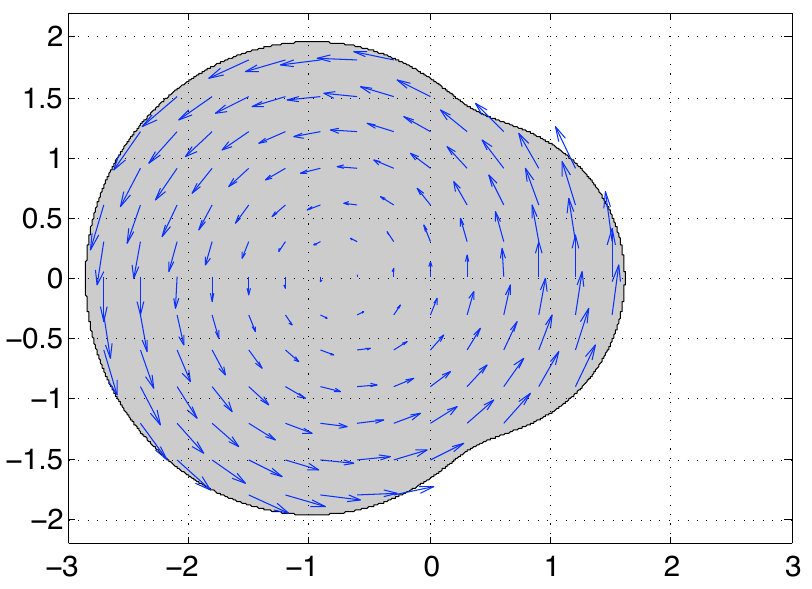}}\\
     \hline
     \end{tabular}
     \caption{\label{fig:1}For both configurations, the stream function is the same. It reads $\psi(r,\theta):=\cos(2\theta)/r^2$ in polar coordinates. The holomorphic potential is $\xi(z):=i/z^2$.}
     \end{figure}
\begin{figure}[h]     
    \centering
     \begin{tabular}{|cc|}
     \hline
     \subfigure
          [$\omega=0.6$, $\rho=0.5$, $\lambda=1.5$ and $\mathbf s_1=(\rho\cos(\pi/4),\rho\sin(\pi/4))$]
     {\includegraphics[width=.35\textwidth]{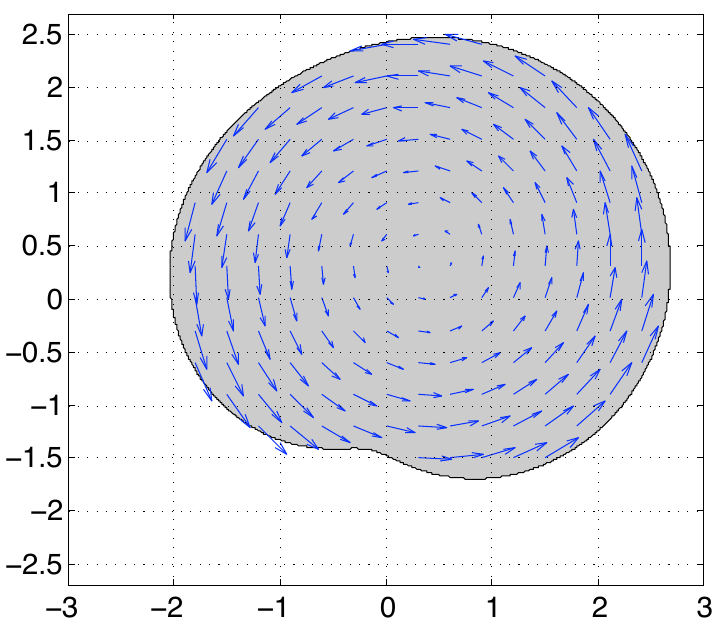}}
     &
     \subfigure
     [$\omega=0.6$, $\rho=0.5$, $\lambda=1.5$ and $\mathbf s_2=(\rho\cos(-\pi/4),\rho\sin(-\pi/4))$]
     {\includegraphics[width=.35\textwidth]{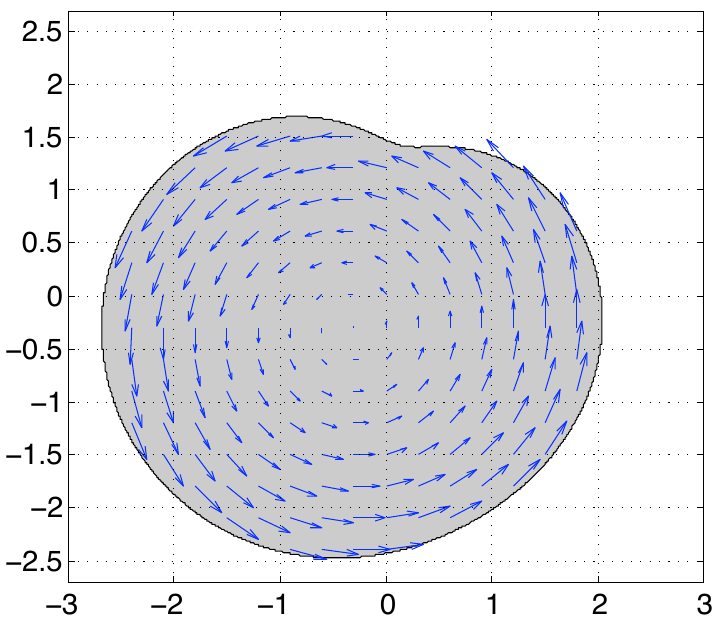}}\\
     \hline
     \end{tabular}
     \caption{\label{fig:2}For both configurations, the stream function and the holomorphic potential are the same as in Figure~\ref{fig:1}.}
     \end{figure}



\begin{figure}[h]     
     \centering
     \begin{tabular}{|c|c|c|}
     \hline
     \subfigure
          [$k=1$]
     {\includegraphics[width=.30\textwidth]{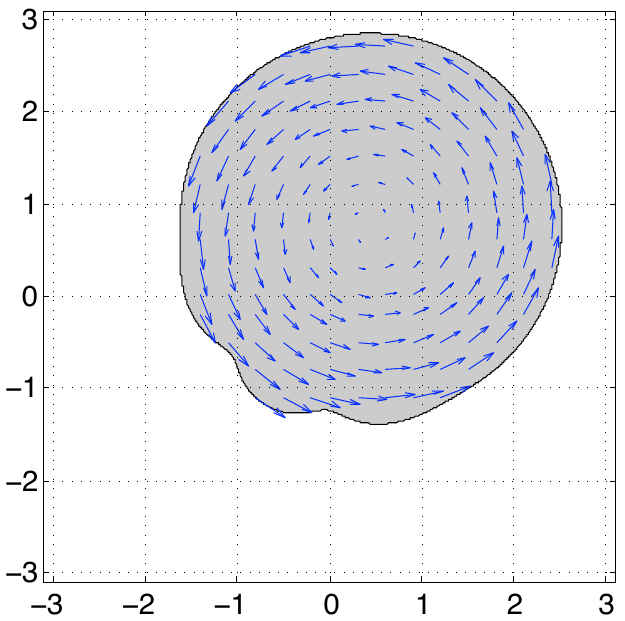}}
     &
     \subfigure
     [$k=2$]
     {\includegraphics[width=.30\textwidth]{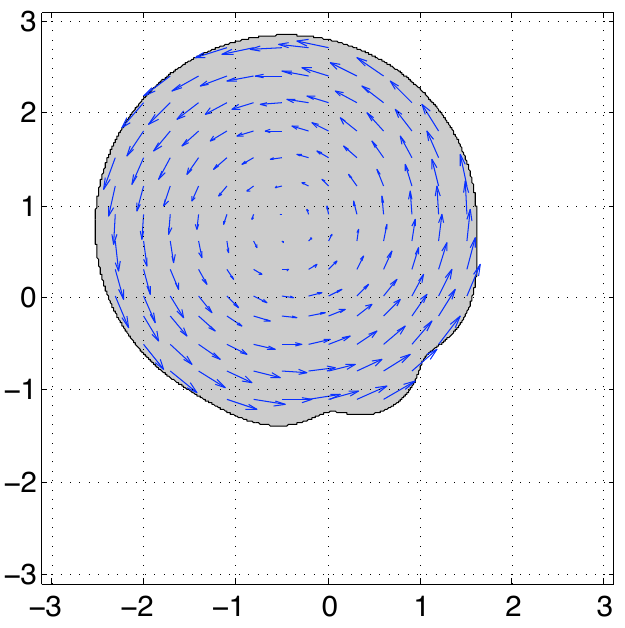}}
     &
      \subfigure
     [$k=3$]
     {\includegraphics[width=.30\textwidth]{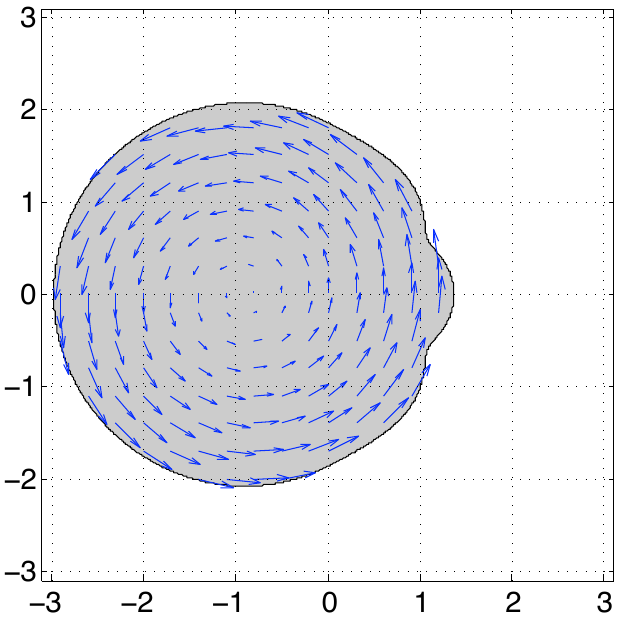}}\\
     \hline
     \subfigure
     [$k=4$]
     {\includegraphics[width=.30\textwidth]{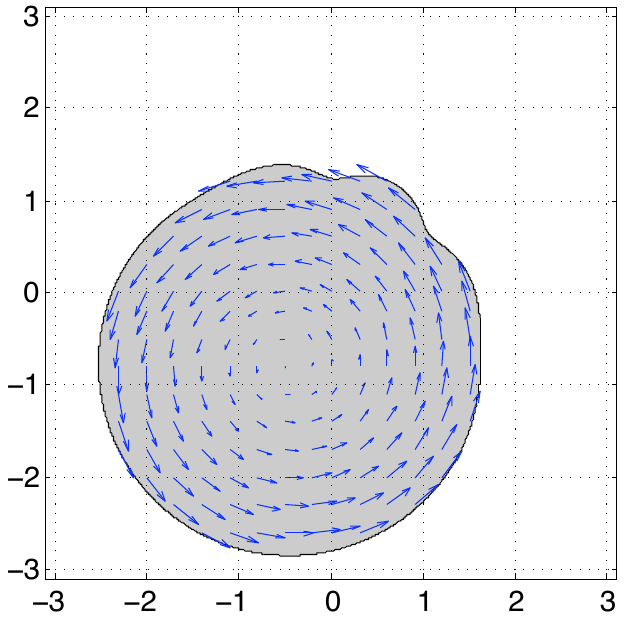}}
     &
     \subfigure
     [$k=5$]
      {\includegraphics[width=.30\textwidth]{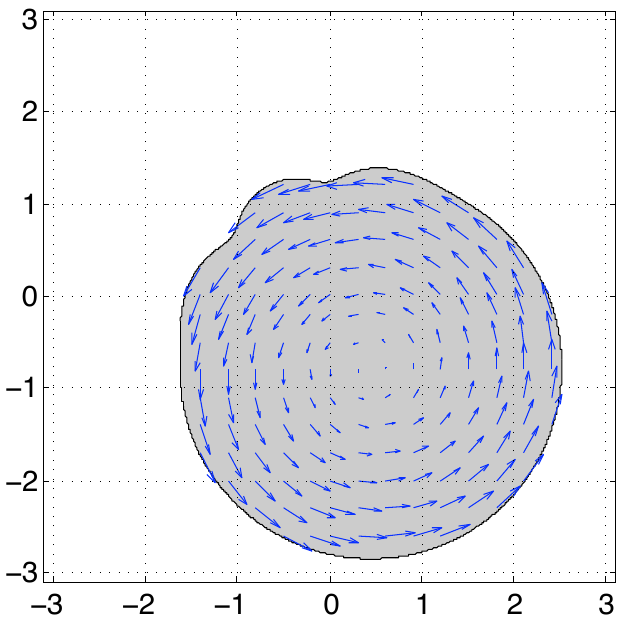}}
     &
     \subfigure
     [$k=6$]
     {\includegraphics[width=.30\textwidth]{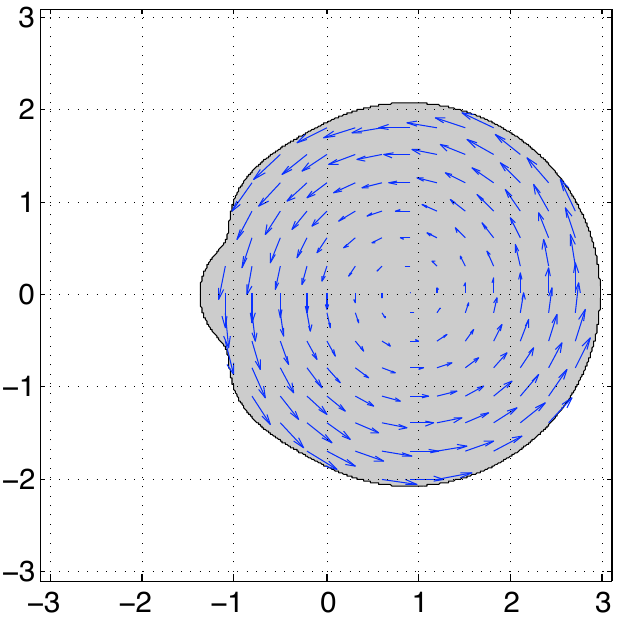}}\\
     \hline
\end{tabular}
\caption{\label{fig:3}The stream function is  $\psi(r,\theta)=\cos(6\theta)/r^6$, the holomorphic potential is $\xi=i/z^6$ and $\omega = 0.7$,
$\rho=0.9$, $\lambda = -2.5$ and 
$\mathbf s_1=(\rho\cos(k\pi/6),\rho\sin(2k\pi/6))$ for $k=1,\ldots,6$.}
\end{figure}

\section{The Complex Potential}
\label{complex:poten}
Before going further, we need to describe the shape $\mathcal S_0$. Actually, for convenience, rather than $\mathcal S_0$ we shall describe $\mathcal F_0:=\mathbf C\setminus \bar{\mathcal S}_0$. Thus, assume that $\mathcal F_0$ is the image by a conformal mapping $f$ of $\Omega:=\mathbf C\setminus \bar D$, the exterior of the unitary disk. For any simply connected shape $\mathcal S_0$ and corresponding domain $\mathcal F_0$, the Riemann Mapping Theorem tells us that $f$ can be written in the form:
\begin{equation}
\label{general:form}
f(z)=c_1 z+c_0+\sum_{k\leq -1}c_k z^k,\quad(z\in\Omega),
\end{equation}
where $c_k\in\mathbf C$ for $k=1$ and all $k\leq -1$ and $c_1\neq 0$. We can assume, without loss of generality, that $c_0=0$. To simplify forthcoming computations, we will also assume that $c_k$ is actually defined for all $k\in\mathbf Z$ and that $c_k=0$ for $k=0$ and $k\geq 2$. We denote $c:=(c_k)_{k\in\mathbf Z}$ the complex sequence of elements $c_k$ and the {\it Area Theorem} (see \cite[Theorem 14.13]{Rudin:1987aa}) tells us that the area of $\mathcal S_0$ is equal to $\pi\sum_{k\leq 1}k|c_k|^2$. Since $\mathcal S_0$ is of finite extent, it means that this sum has to be finite. Actually, we will assume also that $c\in\ell^1(\mathbf C)$, which entails in particular that $f$ is continuous in the closed set $\bar\Omega$.

Such a description allows us to consider a broad set of solids. In particular, the boundary of the solid can be very rough. Degenerate cases  can be considered as well (for instance $\mathcal S_0$ can be a segment modeling a one dimensional beam). 

For any position $\mathbf p:=(R(\alpha),0,\mathbf r)$, we recall that $\mathcal S:=R(\alpha)\mathcal S_0+\mathbf r$ is the actual domain occupied by the solid. Let us introduce then the functions  $\varphi_0(x):=\varphi(R(\alpha)x+\mathbf r)$ and $\psi_0(x):=\psi(R(\alpha)x+\mathbf r)$ which are harmonic (and defined) over the fixed domain $\mathcal F_0$. For any velocity $\mathbf v:=(\omega,\mathbf r,\mathbf w)$ (we choose here $\mathbf s=\mathbf r$), the Dirichlet boundary condition for $\psi$ turns out to be, in complex notation $2i\psi_0:=w_0\bar z-\bar w_0 z+i\omega|z|^2$ where $w_0(z):=w(R(\alpha)z+\mathbf r)$.
We introduce $\zeta$ the holomorphic complex potential of the fluid defined for any $z\in\Omega$ by $\zeta(z)=\varphi_0(f(z))+i\psi_0(f(z))$.
Since $\bar z=1/z$ on $\partial\Omega$, we get the identity $2i\psi_0(f(z))=-\bar w f(z)+w\bar f(1/z)+i\omega f(z)\bar f(1/z)$. For any $z\in\partial\Omega$, we have also
$\bar f(1/z)=\sum_{k\in{\mathbf Z}}\bar c_{-k}z^k=\sum_{k\in{\mathbf Z}}\check c_{k}z^k$ and 
$f(z)\bar f(1/z)=\sum_{k\in{\mathbf Z}}(\check c\ast c)_k z^k$.
So we get:
\begin{equation}
\label{above}
2i\psi_0(f(z))=\sum_{k\in{\mathbf Z}}[-\bar w_0c_k+w_0\check c_{k}+i\omega(\check c\ast c)_k] z^k,\quad(z\in\partial\Omega).
\end{equation}
According to \cite[Chap. IX, \S 9.63]{Milne-Thomson:1960aa}, we keep only the negative powers in \eqref{above} to get the expression of $\zeta$. Defining the coefficients $\zeta_k(w_0,\omega):=[-\bar w_0c_k+w_0\check c_{k}+i\omega(\check c\ast c)_k]$ for all $k\leq -1$, we obtain:
\begin{equation}
\label{potential}
\zeta(z)=\sum_{k\leq -1} \zeta_k(w_0,\omega) z^k,\quad(z\in\Omega).
\end{equation}
Eventually, the expression of the measured complex potential,  defined in $\mathcal F$, is:
\begin{equation}
\xi(z)=\zeta(f^{-1}((z-r)e^{-i\alpha})),\quad(z\in\mathcal F).
\end{equation}
According to our rule of notation, we introduce as well 
\begin{equation}
\xi_0(z)=\varphi_0(z)+\psi_0(z)=\zeta(f^{-1}(z)),\quad(z\in\mathcal F_0).
\end{equation}
\section{Stealth Rigid Solids}
\label{sec:stealth}
In this Section we wish to determine all the possible shapes and configurations of solids for which the complex potential $\xi$ is identically null. Such a displacement will be termed {\it stealth}.
\begin{theorem}
\label{theo:stelth}
The only solids $\mathcal S$ that can undergo stealth motions in a fluid are:
\begin{itemize}
\item Disks rotating about their centers;
\item Arc of circles and segments with velocity field everywhere tangent to $\mathcal S$. 
\end{itemize}
\end{theorem}
The arc of circles and segments are one dimensional solids and can be considered as degenerated cases.
\begin{proof}
Let us assume that $\xi=0$. Then we have also $\zeta=0$ which means that $-2\Im(\bar u f(z))+\omega|f(z)|^2=0$ for all $z\in\partial\Omega$. If $\omega\neq 0$, some easy computations tell us that for all $z\in\partial\Omega$, $f(z)$ belongs to the circle of center $iw_0/2\omega$ and radius $|w_0|/(2|\omega|)$. Since $f$ is an homeomorphism from $\partial\Omega$ onto $f(\partial\Omega)$, $f(\partial\Omega)$ is a connected compact subset of this circle. 
\begin{itemize}
\item If $f(\partial\Omega)$ is the complete circle, it means that $c_1=1$ and $c_k=0$ for all $k\neq 1$. In this case, since $\zeta_1=0$ and $(\check c\ast c)_1=0$, we deduce that $w_0=0$ and hence that the circle is just rotating about its center.
\item
Up to a translation and a rotation, all the conformal mappings that map the circle onto an arc of circle have the form $f(z)=z+(1-h^2)/(z+ih)$ where $h$ is any real number such that $0<h<1$. We can put $f$ into the general form \eqref{general:form} by setting: $c_1=1$, $c_{-1}=1-h^2$ and $c_k=(1-h^2)(-ih)^{-k-1}$ for all $k\leq -2$. Some simple computations lead to:
$(\check c\ast c)_{-1}=-ihc_{-1}$ and $(\check c\ast c)_{k}=(ih^{-1}-ih)c_k$ for all $k\leq -2$. Substituting these expressions into \eqref{potential} and writing that $\zeta_k(w_0,\omega)=0$ for all $k\leq -1$ we obtain the same equation for all $k$ which yields the relation: $w_0=\omega(h-1/h)$. We can then easy prove that this motion corresponds to the case where the velocity field is tangent to the solid.
\end{itemize}
Let us assume now that $\omega$ is zero (and $w_0\neq 0$). In this case, we deduce with \eqref{potential}, that $c_k=0$ for all $k\leq -2$. For $k=-1$, we get $c_{-1}=\bar c_1{w_0}/{\bar w_0}$.
We set $w_0=Re^{i\theta}$, $c_1={\tilde R}e^{i\beta}$ and we rewrite $f$ in the form:
$$f(z)={\tilde R}\left[e^{i\beta} z+e^{i(-\beta+2\theta)}/z\right]=2{\tilde R}e^{i\theta}\left[e^{i(\beta-\theta)}z+e^{-i(\beta-\theta)}/z\right].$$
We seek the image of the unitary circle by $f$. We specify $z=e^{it}$ with $t\in{\mathbf R}/2\pi$ and we get $f(e^{it})=2 {\tilde R}e^{i\theta}\cos(\beta-\theta+t)$.
So the image of the unitary circle is the segment $[-{\tilde R},{\tilde R}]$ turned by an angle $\theta$. The velocity $w_0$ is collinear to the segment.
\end{proof}
\section{Detection of a Moving Ellipse}
\label{detection:ellipse}
When $\mathcal S_0$ is an ellipse, the function $f$ has the form $f(z)=(a+b)z/2+(a-b)/2z$, where $a, b\in\mathbf R_+$, $a>b>0$. 
We can now give the proof of Proposition~\ref{prop:ellipse}.
\begin{proof}
First, we can explicitly compute the inverse function 
\begin{equation}
\label{inverse}
f^{-1}(z)=\frac{z}{(a+b)}\Big(1+\sqrt{1-\frac{(a^2-b^2)}{z^2}}\Big),\quad (z\in\mathcal F_0).
\end{equation}
In this expression, $-\sqrt{a^2-b^2}$ and $\sqrt{a^2-b^2}$ are branch points and the function is holomorphic everywhere but on the segment $[-\sqrt{a^2-b^2},\sqrt{a^2-b^2}]$ which is a branch cut.
Next, we get:
$$\zeta(z)=\Big[-\bar w_0 \frac{a-b}{2}+w_0\frac{a+b}{2}\Big]\frac{1}{z}+i\omega \frac{a^2-b^2}{4}\frac{1}{z^2},\quad(z\in\Omega),$$
and then:
$$
\xi(z)=\frac{[-(a^2-b^2)\bar w_0+(a+b)^2w_0]e^{i\alpha}}{2(z-r)\Big[1+\sqrt{1-\displaystyle\frac{(a^2-b^2)e^{2i\alpha}}{(z-r)^2}}\Big]}
+\frac{ i (a^2-b^2)(a+b)^2e^{2i\alpha}\omega}{4(z-r)^2\Big[1+\sqrt{1-\displaystyle\frac{(a^2-b^2)e^{2i\alpha}}{(z-r)^2}}\Big]^2},\quad(z\in\mathcal F).
$$
Observe that, due to the symmetry of the ellipse, we can change $\alpha$ into $\alpha+\pi$ and accordingly $w_0$ into $-w_0$ without changing the expression of $\xi$. 
The potential $\xi$ is holomorphic everywhere but on the branch cut $[r-\sqrt{a^2-b^2}e^{i\alpha},r+\sqrt{a^2-b^2}e^{i\alpha}]$. So if we know thoroughly $\xi$, we can determine the location of the branch points $r-\sqrt{a^2-b^2}e^{i\alpha}$ and $r+\sqrt{a^2-b^2}e^{i\alpha}$ and hence also the position of the center $r$ and the orientation $\alpha$ (up to $\pi$ only). 
We compute next the limit $\mu:=\lim_{|z|\to +\infty} e^{-i\alpha}\xi(z)z/(a+b)=[-(a-b)\bar w_0+(a+b)w_0]/4$ and we deduce the expression of $w_0$, namely:
$w_0=(\mu+\bar\mu)/b+(\mu-\bar\mu)/a$. The only remaining unknown quantity $\omega$ is next easily obtained, following the same idea.
\end{proof}
\begin{figure}[h]     
    \centering
     \begin{tabular}{|cc|}
     \hline
     \subfigure
          [Level sets of the potential function $\varphi_0=\Re(\xi_0)$]
     {\includegraphics[height=.44\textwidth]{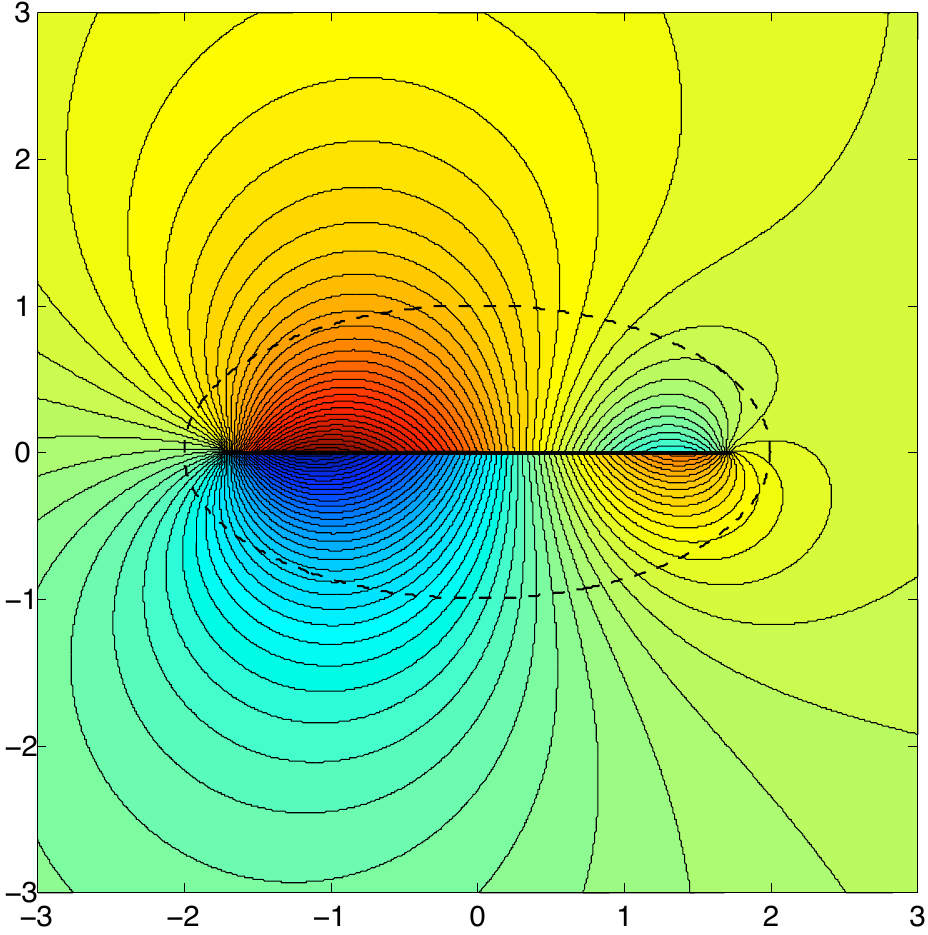}}
     &
     \subfigure
     [Level sets of the stream function $\psi_0=\Im(\xi_0)$]
     {\includegraphics[height=.44\textwidth]{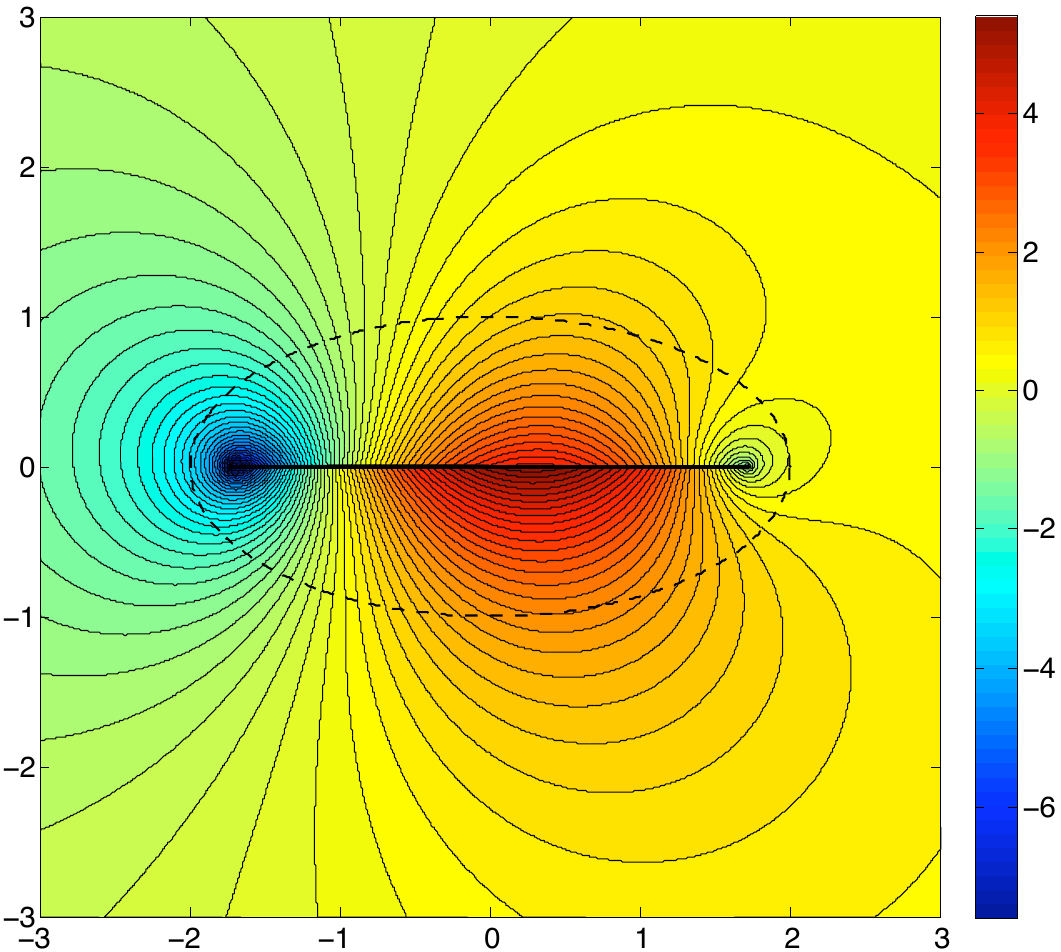}}\\
     \hline
     \end{tabular}
     \caption{\label{fig:22}Level sets of the holomorphic potential $\xi_0$, for $a=2$, $b=1$, $w_0=e^{i\pi/3}$ and $\omega=-2$. The boundary of the ellipse (dashed line) can hardly be directly detected but the branch points $-\sqrt{3}$ and $\sqrt{3}$ are clearly identifiable.}
     \end{figure}

\section{Detection: General Case}
\label{singu}
\subsection{Singularities of the holomorphic potential}
\label{subsect:singul}
In the preceding example, the branch points of the potential $\xi$ played a crucial role in determining the position of the solid in the fluid. 
Notice that the existence of these points did not depend on the configuration but only on the shape of the solid (they came from the definition \eqref{inverse} of the inverse function $f^{-1}$ and were subsequently just translated and rotated according to the position). We shall prove that this result can be generalized to any solid: there is no singularity in the potential function, that does not come from the conformal mapping $f^{-1}$ (but unfortunately, the potential function may have less singular points than the function $f^{-1}$). Let us make this statement precise:

\begin{definition}[Analytic continuation]
An holomorphic function $\tilde\xi$ (respect. $\tilde \xi_0$) defined in a connected open set $\widetilde{\mathcal F}$ (respect. $\widetilde{\mathcal F}_0$) containing $\mathcal F$ (respect. $\mathcal F_0$) is called an analytic continuation of $\xi$ (respect. $\xi_0$) when $\tilde\xi=\xi$ in $\mathcal F$ (respect. $\tilde\xi_0=\xi_0$ in $\mathcal F_0$). 
\end{definition}
There may exist several analytic continuations of $\xi$ that do not coincide everywhere. Assume that $\tilde\xi_1$ and $\tilde\xi_2$ are two such functions defined respectively on $\widetilde{\mathcal F}_1$ and $\widetilde{\mathcal F}_2$. So the Analytic Continuation Theorem  ensures only that $\tilde\xi_1=\tilde\xi_2$ on the connected component of $\widetilde{\mathcal F}_1\cap\widetilde{\mathcal F}_2$ containing $\mathcal F$. In Section~\ref{detection:ellipse} for instance, we can not choose where the branch points are, but there are many different possible choices for the branch cut, each one corresponding to a different analytic continuation of $\xi$.

Assume that for some potential function $\xi$, there exists an analytic continuation $\tilde\xi$ such that $\widetilde{\mathcal F}=\mathbf C$. Since, by construction, $\xi(z)$ tends to $0$ as $|z|$ goes to infinity, $\tilde\xi$ is a bounded entire function. According to Liouville's Theorem, this function is constant, equal to 0. This case was treated in Section~\ref{sec:stealth} and is possible only for solids listed in Theorem~\ref{theo:stelth}.
For all of the other solids and for any analytic continuation $\tilde\xi$, there exists at least one point, located inside the solid, which does not belong to $\widetilde{\mathcal F}$. This very simple observation allows one to locate the solid in a very first approximation. 

Let us now prove that the singularities of $\tilde\xi$ come from the singularities of $f^{-1}$.
\begin{prop}
If there exists  an analytic continuation $\tilde g$ of $f^{-1}$ defined on an open connected set $\widetilde{\mathcal F}_0$ containing $\mathcal F_0$, then for any configuration $(\mathbf p,\mathbf v)\in\mathcal P\times\mathcal V$, there exists an analytic continuation $\tilde\xi$ of $\xi$ defined on $\widetilde{\mathcal F}:=\mathbf p \big(\widetilde{\mathcal F}_0\setminus \tilde g^{-1}(\{0\})\big)$.
\end{prop}
In this proposition, the notation $\tilde g^{-1}(\{0\})$ stands for the preimage of $\{0\}$ under $\tilde g$ and does not mean that $\tilde g$ is invertible. Since $\tilde g$ is holomorphic, the set $\tilde g^{-1}(\{0\})$ consists only in isolated points and $\widetilde{\mathcal F}_0\setminus \tilde g^{-1}(\{0\})$ is still connected and still contains ${\mathcal F}_0$.
\begin{proof}
For all $z\in\partial D$, we can rewrite $\zeta$ in the form:
\begin{multline}
\zeta(z)=-\bar w_0(f(z)-c_1z)+w_0{\bar c_1}{z}^{-1}+i\omega\Big[\bar c_1z^{-1}(f(z)-c_1z)+\\
\sum_{k\geq 1}\bar c_{-k}z^k\Big(f(z)-\sum_{-k\leq j\leq 1}c_jz^j\Big)\Big].
\end{multline}
Expanding the right hand side and recombining terms, we get the identity:
$$\zeta(z)=-\bar w_0f(z)+c_1\bar w_0z+u\bar c_1z^{-1}+i\omega\Big[f(z)\bar f(z^{-1})-H_1(z)\Big],\quad(z\in\partial D),$$
where $H_1(z):=\sum_{j\geq 0}(\check c\ast c)_j z^j$ and $\bar f(z):=\bar c_1z+\sum_{k\leq -1}\bar c_kz^k$.
Classical results for the convolution product ensure that this series is uniformly convergent for $|z|\leq 1$ since $\|\check c\ast c\|_{\ell^1(\mathbf C)}\leq \|\check c\|_{\ell^1(\mathbf C)}\|c\|_{\ell^1(\mathbf C)}$. 
We next obtain that:
$$\zeta(f^{-1}(z))=-\bar w_0z+c_1\bar w_0f^{-1}(z)+u\bar c_1/f^{-1}(z)+i\omega\Big[z\bar f(1/f^{-1}(z))-H_1(f^{-1}(z))\Big],\quad(z\in\partial\mathcal S_0).$$
Let $\tilde g$ by any analytic continuation of $f^{-1}$ (not necessary invertible), defined in an open set $\widetilde{\mathcal F}_0$.  It can be split into three parts: $\widetilde{\mathcal F}_0^+:=
\{z\in\widetilde{\mathcal F}_0\,:\,|\tilde g(z)|> 1\}$, $\widetilde{\mathcal F}_0^-:=
\{z\in\widetilde{\mathcal F}_0\,:\,0<|\tilde g(z)|< 1\}$ and $\widetilde{\mathcal F}_0^1:=\{z\in\mathbf C\,:\,\tilde g(z)=1\}$. We can next define:
\begin{alignat*}{3}
\tilde\xi_0(z):=&-\bar w_0z+c_1\bar w_0\tilde g(z)+w_0\bar c_1/\tilde g(z)+i\omega\Big[z\bar f(1/\tilde g(z))-H_1(\tilde g(z))\Big],&\quad& z\in\widetilde{\mathcal F}_0^-\cup\widetilde{\mathcal F}_0^1,\\
\tilde\xi_0(z):=&-\bar w_0z+c_1\bar w_0\tilde g(z)+w_0\bar c_1/\tilde g(z)+i\omega\Big[H_2(\tilde g(z))\Big],&&z\in\widetilde{\mathcal F}_0^+,
\end{alignat*}
where $H_2(z):=\sum_{j\leq -1}(\check c\ast c)_j z^j$
is uniformly convergent for $|z|\geq 1$.
We deduce that the function $\tilde\xi_0$ is holomorphic in $\widetilde{\mathcal F}_0^-$ and in $\widetilde{\mathcal F}_0^+$ and continuous in $\widetilde{\mathcal F}_0\setminus \tilde g^{-1}(\{0\})$. Let $z_0\in\widetilde{\mathcal F}_0^1$ and denote $z_1:=\tilde g(z_0)$. Since $\tilde g$ is holomorphic at the point $z_0$, there exists $R>0$, $n\geq 1$ and a function $h$ holomorphic in the disk $D(0,R)$ such that $h(0)\neq 0$ and 
$$\tilde g(z)=z_1+(z-z_0)^nh(z-z_0),\quad z\in D(0,R).$$
This identity allows us to describe the set $\widetilde{\mathcal F}_0^1$ nearby the point $z_0$. For instance, if $n=3$, and for $R$ small enough, we get something like in Figure~\ref{conf_map} where $z_1=\tilde g(z_0)$, $\cup_{j=1}^3 A_j^+=D(z_0,R)\cap \widetilde{\mathcal F}_0^+$, $\cup_{j=1}^3 A_j^-=D(z_0,R)\cap \widetilde{\mathcal F}_0^-$ and the curves radiating from $z_0$ correspond to the set $D(z_0,R)\cap\widetilde{\mathcal F}_0^1$. Except maybe at the point $z_0$, the boundaries shared by the regions $A^+_j$ and $A^-_j$ are smooth. The function $\tilde g$ maps $A_1^+$ onto $U^+$ (an open set located outside the unitary disk) and $A_1^-$ onto $U^-$ (an open set located inside the unitary disk). Furthermore, the function $\tilde g|_{A^+\cup A_1^-}:A_1^+\cup A_1^-\to U^-\cup U^+$ is a conformal mapping. Consider next a conformal mapping $\phi$ which maps a neighbourhood of $z_1$ onto a neighbourhood of $0$ as in Figure~\ref{conf_map} and such that the image of the unitary circle is the imaginary axis. We can then apply \cite[Theorem 16.8]{Rudin:1987aa}: the function $\xi_0\circ\tilde g^{-1}\circ\phi^{-1}$ is holomorphic on both sides of the imaginary axis and continuous across this boundary, so it is holomorphic over the whole domain. We deduce that $\xi_0$ is holomorphic across the boundary between $A_1^+$ and  $A_1^-$. We can repeat this process with the domains $A_1^+\cup A_2^-$, $A_2^+\cup A_2^-$ and so on. Finally, we obtain that $\xi_0$ is holomorphic on the whole disk $D(z_0,R)$ except maybe at the point $z_0$. But once more, the continuity of the function $\xi_0$ does not allow this possibility.
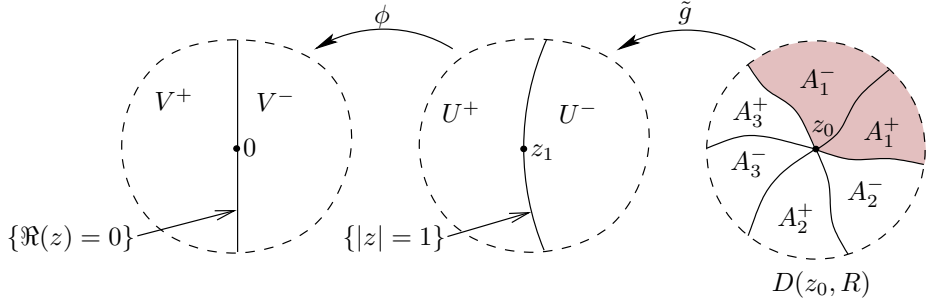
\begin{figure}
\centerline{
\begin{picture}(0,0)%
\includegraphics{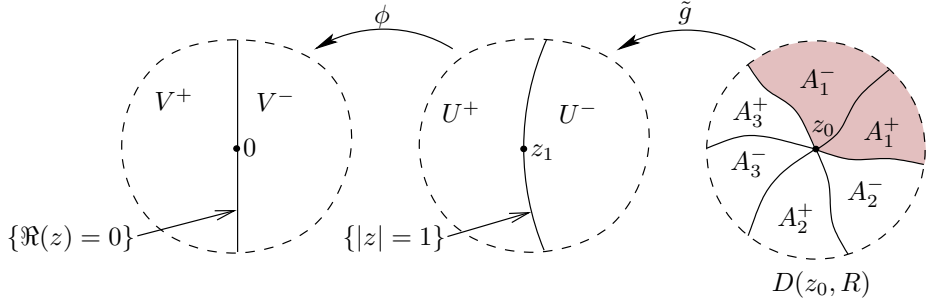}%
\end{picture}%
\setlength{\unitlength}{3947sp}%
\begingroup\makeatletter\ifx\SetFigFont\undefined%
\gdef\SetFigFont#1#2#3#4#5{%
  \fontfamily{#3}\fontseries{#4}\fontshape{#5}%
  \selectfont}%
\fi\endgroup%
\begin{picture}(5066,1906)(2680,-5513)
\put(7368,-4879){\makebox(0,0)[b]{\smash{{\SetFigFont{10}{12.0}{\rmdefault}{\mddefault}{\updefault}{\color[rgb]{0,0,0}$A_2^-$}%
}}}}
\put(7469,-4500){\makebox(0,0)[b]{\smash{{\SetFigFont{10}{12.0}{\rmdefault}{\mddefault}{\updefault}{\color[rgb]{0,0,0}$A_1^+$}%
}}}}
\put(7067,-4172){\makebox(0,0)[b]{\smash{{\SetFigFont{10}{12.0}{\rmdefault}{\mddefault}{\updefault}{\color[rgb]{0,0,0}$A_1^-$}%
}}}}
\put(6928,-5039){\makebox(0,0)[b]{\smash{{\SetFigFont{10}{12.0}{\rmdefault}{\mddefault}{\updefault}{\color[rgb]{0,0,0}$A_2^+$}%
}}}}
\put(6646,-4374){\makebox(0,0)[b]{\smash{{\SetFigFont{10}{12.0}{\rmdefault}{\mddefault}{\updefault}{\color[rgb]{0,0,0}$A_3^+$}%
}}}}
\put(7024,-4448){\makebox(0,0)[lb]{\smash{{\SetFigFont{10}{12.0}{\rmdefault}{\mddefault}{\updefault}{\color[rgb]{0,0,0}$z_0$}%
}}}}
\put(6635,-4688){\makebox(0,0)[b]{\smash{{\SetFigFont{10}{12.0}{\rmdefault}{\mddefault}{\updefault}{\color[rgb]{0,0,0}$A_3^-$}%
}}}}
\put(5261,-4603){\makebox(0,0)[lb]{\smash{{\SetFigFont{10}{12.0}{\rmdefault}{\mddefault}{\updefault}{\color[rgb]{0,0,0}$z_1$}%
}}}}
\put(6232,-3742){\makebox(0,0)[b]{\smash{{\SetFigFont{10}{12.0}{\rmdefault}{\mddefault}{\updefault}{\color[rgb]{0,0,0}$\tilde g$}%
}}}}
\put(4704,-4387){\makebox(0,0)[lb]{\smash{{\SetFigFont{10}{12.0}{\rmdefault}{\mddefault}{\updefault}{\color[rgb]{0,0,0}$U^+$}%
}}}}
\put(5438,-4392){\makebox(0,0)[lb]{\smash{{\SetFigFont{10}{12.0}{\rmdefault}{\mddefault}{\updefault}{\color[rgb]{0,0,0}$U^-$}%
}}}}
\put(3455,-4597){\makebox(0,0)[lb]{\smash{{\SetFigFont{10}{12.0}{\rmdefault}{\mddefault}{\updefault}{\color[rgb]{0,0,0}$0$}%
}}}}
\put(4327,-3764){\makebox(0,0)[b]{\smash{{\SetFigFont{10}{12.0}{\rmdefault}{\mddefault}{\updefault}{\color[rgb]{0,0,0}$\phi$}%
}}}}
\put(3538,-4315){\makebox(0,0)[lb]{\smash{{\SetFigFont{10}{12.0}{\rmdefault}{\mddefault}{\updefault}{\color[rgb]{0,0,0}$V^-$}%
}}}}
\put(2899,-4310){\makebox(0,0)[lb]{\smash{{\SetFigFont{10}{12.0}{\rmdefault}{\mddefault}{\updefault}{\color[rgb]{0,0,0}$V^+$}%
}}}}
\put(4728,-5159){\makebox(0,0)[rb]{\smash{{\SetFigFont{10}{12.0}{\rmdefault}{\mddefault}{\updefault}{\color[rgb]{0,0,0}$\{|z|=1\}$}%
}}}}
\put(7087,-5449){\makebox(0,0)[b]{\smash{{\SetFigFont{10}{12.0}{\rmdefault}{\mddefault}{\updefault}{\color[rgb]{0,0,0}$D(z_0,R)$}%
}}}}
\put(2776,-5161){\makebox(0,0)[rb]{\smash{{\SetFigFont{10}{12.0}{\rmdefault}{\mddefault}{\updefault}{\color[rgb]{0,0,0}$\{\Re(z)=0\}$}%
}}}}
\end{picture}
}
\caption{\label{conf_map}The conformal mapping $\tilde g$ around the point $z_0$ for $n=3$.}
\end{figure}
\end{proof} 
As already mentioned, the potential function can make some singularities of $f^{-1}$ to vanish. This is illustrated by the example of Figure~\ref{conf_map_p}.
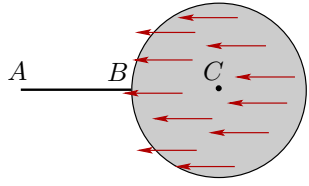
\begin{figure}
\centerline{
\begin{picture}(0,0)%
\includegraphics{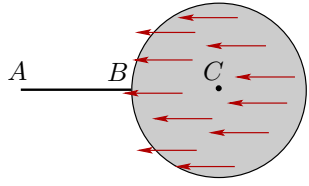}%
\end{picture}%
\setlength{\unitlength}{3947sp}%
\begingroup\makeatletter\ifx\SetFigFont\undefined%
\gdef\SetFigFont#1#2#3#4#5{%
  \reset@font\fontsize{#1}{#2pt}%
  \fontfamily{#3}\fontseries{#4}\fontshape{#5}%
  \selectfont}%
\fi\endgroup%
\begin{picture}(1900,1110)(3797,-5258)
\put(5041,-4636){\makebox(0,0)[lb]{\smash{{\SetFigFont{10}{12.0}{\rmdefault}{\mddefault}{\updefault}{\color[rgb]{0,0,0}$C$}%
}}}}
\put(3812,-4641){\makebox(0,0)[lb]{\smash{{\SetFigFont{10}{12.0}{\rmdefault}{\mddefault}{\updefault}{\color[rgb]{0,0,0}$A$}%
}}}}
\put(4437,-4642){\makebox(0,0)[lb]{\smash{{\SetFigFont{10}{12.0}{\rmdefault}{\mddefault}{\updefault}{\color[rgb]{0,0,0}$B$}%
}}}}
\end{picture}%
}
\caption{\label{conf_map_p}The solid consists in a disk of center $C$ with a segment $[A,B]$. When this solid is moving to the left, parallel to the segment, one can easily check that the potential coincides with the potential of the disk. Although any analytic continuation of $f^{-1}$ has singularities at $A$ and $B$, the complex potential does not see these points. In this case, the singularities do not allow one to determine the orientation of the solid.}
\end{figure}
\subsection{Asymptotic expansion of the holomorphic potential}
In this section, we shall compute the asymptotic expansion of $\xi$ in terms of the geometrical data of $\mathcal S_0$ and the configuration $(\mathbf p,\mathbf v)\in\mathcal P\times \mathcal V$. As explained in the preceding section, we know that if $\mathcal S_0$ is not one of the solid listed in Theorem~\ref{theo:stelth}, the potential function admits no analytic continuation on the whole complex plane. It allows one to deduce  approximately where the solid is. For all $\nu\in\mathbf C$,  we can next consider $\Gamma$, a contour large enough to encircle the solid and the point $\nu$. The contour $\widetilde\Gamma:=f(e^{-i\alpha}(\Gamma-r))$ encircles the unitary disk. According to the expression \eqref{a:priori} of the potential as a Laurent series, we obtain that, for all $n\geq 1$:
\begin{align*}
\lambda_n(\nu)=\frac{1}{2i\pi}\oint_{\Gamma} \xi(z)(z-\nu)^{n-1}{\rm d}z&=\frac{e^{i\alpha}}{2i\pi}\oint_{\widetilde\Gamma} \zeta(z)\left(e^{i\alpha}f(z)+r-\nu\right)^{n-1}f'(z){\rm d}z\\
&=\frac{1}{n}\frac{1}{2i\pi}\oint_{\widetilde\Gamma}\zeta(z)\frac{d}{dz}\left(e^{i\alpha}f(z)+r-\nu\right)^{n}{\rm d}z\\
&=-\frac{1}{n}\frac{1}{2i\pi}\oint_{\widetilde\Gamma}\zeta'(z)\left(e^{i\alpha}f(z)+r-\nu\right)^{n}{\rm d}z\\
&=-\frac{1}{n}\frac{1}{2i\pi}\sum_{k=0}^{n}\binom{n}{k}e^{ik\alpha} (r-\nu)^{n-k}\oint_{\widetilde\Gamma} \zeta'(z)f(z)^k{\rm d}z.
\end{align*}
If we define the complex sequence $d:=(d_k)_{k\in\mathbf Z}$ by $d_k:=(k+1)\zeta_{k+1}$ for all $k\leq -1$ and $d_k=0$ for $k\geq 0$, we obtain: 
$$\lambda_n(\nu)=-\frac{1}{n}\sum_{k=1}^{n}\binom{n}{k}e^{ik\alpha} (r-\nu)^{n-k}(d\ast c^k)_{-1},\quad (\nu\in\mathbf C).$$
We can rewrite the last term:
\begin{align*}
(d\ast c^k)_{-1}&=\mespn\sum_{i_1+\ldots+i_{k+1}=-1}c_{i_1}\ldots c_{i_k}d_{i_{k+1}}\\
&=-A_{k}\bar w_0+B_{k}w_0+i\omega C_{k},
\end{align*}
where 
\begin{alignat*}{3}
A_k&:=\mespn\sum_{i_1+\ldots+i_{k+1}=0\atop i_1\leq -1}\mespn i_1c_{i_1}\ldots c_{i_{k+1}},&\quad&
B_k&:=\mespn\sum_{i_1+\ldots+i_{k+1}=0\atop i_1\leq -1}\mespn i_1\bar c_{-i_1}c_{i_2}\ldots c_{i_{k+1}},\\
C_k&:=\mespn\sum_{i_1+\ldots+i_{k+2}=0\atop i_1+i_2\leq -1}\mespn (i_1+i_2)\bar c_{-i_1}c_{i_2}\ldots c_{i_{k+2}}.
\end{alignat*}
In the following, to simplify the notation, we consider the quantities $\mathcal A_k:=-A_k/k$, $\mathcal B_k:=-B_k/k$ and $\mathcal C_k:=-C_k/k$. Indeed, we get, for all $n\geq 1$:
\begin{equation}
\label{system:1}
\lambda_n(\nu)=\sum_{k=1}^n\binom{n-1}{k-1}e^{ik\alpha}(r-\nu)^{n-k}\Big[-\mathcal A_{k}\bar w_0+\mathcal B_{k}w_0+i\omega\mathcal C_{k}\Big].
\end{equation}
The problem of detection can now be reformulated as a purely algebraic problem: the complex sequence $(\lambda_j(\nu))_{j\geq 1}$ being given for all $\nu\in\mathbf C$, as well as the complex numbers $\mathcal A_k$, $\mathcal B_K$ and $\mathcal C_k$ ($k\geq 1)$, can we solve the infinite nonlinear system of equations \eqref{system:1} and find the values of $r,\alpha,w_0$ and $\omega$? According to the results of both Section~\ref{sec:contre_exx} and Section~\ref{sec:stealth}, we already know that there exist cases (namely, coefficients $(c_k)_{k\in\mathbf Z})$ for which the answer is negative.

In order to rewrite this infinite set of equations in a convenient short form, we introduce some linear operators: let us denote by $N$ any positive integer and define $\mathcal G_N:\mathbf R^3\to \mathbf C^{N}$  by $(\mathcal G_NU)_{k}:=(-\mathcal A_k+\mathcal B_k)U_1+i(\mathcal A_k+\mathcal B_k)U_2+i\mathcal C_kU_3$ for all $U=(U_1,U_2,U_2)^T\in\mathbf R^3$ and all $1\leq k\leq N$. We define  $D_N:\mathbf C^{N}\to\mathbf C^{N}$ and $S_N:\mathbf C^{N}\to\mathbf C^{ N}$   
as well, by respectively $(D_NZ)_k:=kZ_{k}$ ($1\leq k\leq N$) and $(S_NZ)_1:=0$ and $(S_NZ)_k:=Z_{k-1}$ $(2\leq k\leq N)$ for all $Z:=(Z_1,\ldots,Z_N)\in\mathbf C^N$. The first $N$ equations \eqref{system:1} can now be rewritten as:
\begin{subequations}
\label{main:eq}
\begin{align}
\Lambda_N(\nu)&=e^{\log(r-\nu)D_N}e^{S_ND_N}e^{-\log(r-\nu)D_N}e^{i\alpha D_N}\mathcal G_NU,\\
&=\Theta_N(r-\nu,\alpha)\mathcal G_NU,
\end{align}
\end{subequations}
where $U:=(\Re(w_0),\Im(w_0),\omega)^T$ and $\Lambda_N(\nu):=(\lambda_1(\nu),\ldots,\lambda_N(\nu))^T$. The operator $e^{S_ND_N}$ is lower triangular and the identity $(e^{S_ND_N})_{k,n}=\binom{n-1}{k-1}$ for all $1\leq k\leq n\leq N$, not so obvious, can be found in \cite{Call:1993aa}. Considering the expressions \eqref{main:eq}, it is worth noting that:
\begin{itemize}
\item In \eqref{main:eq}, the coefficients $\lambda_j(\nu)$ in the asymptotic expansion of $\xi$ are obtained by applying to the vector $U$ (the {\it velocity}) first the operator $\mathcal G_N$ encapsulating the information relating to the geometry of the solid and next the operator $\Theta_N(r-\nu,\alpha)$ depending only on the position.
\item The linear operator $\mathcal G_N$ depends on the complex sequence $c$ only, i.e. on the shape of the solid. Moreover, the complex quantities $\mathcal A_k-c_{-k}c_1^k$ ($k\geq 1$), $\mathcal B_{k+1}-c_{-k}\bar c_1c_1^k$ ($k\geq 2$) and $\mathcal C_{k-1}-c_{-k}\bar c_{-1}c_1^{k-1}$ ($k\geq 1$) do not depend on $c_{-n}$ for all $n\geq k$. In other words, for all $N\geq 1$, $\mathcal G_N$ depends on $c_1, c_{-1}, c_{-2},\ldots,c_{-N-1}$ only. We deduce:
\begin{prop}
Let $\mathcal S_0^1$ and $\mathcal S_0^2$ be two shapes described by means of the complex sequences $(c^1_k)_{k\geq 1}$ and $(c^2_k)_{k\geq 1}$, such that $c^1_k=c^2_k$ for all $1\leq k\leq N$. Then, if both solids have the same configuration, their complex potentials will have the same asymptotic expansion up to the order $N-1$. 
\end{prop}
\item The solutions $(r_1,r_2,\alpha,\Re(w_0),\Im(w_0),\omega)^T$ of all of the equations \eqref{main:eq} (for all $N\geq 1$), form a sub-analytic set of $\mathbf R^6$ (because $\Theta_N$ is analytic in $r_1,r_2$ and $\alpha$). This subanalytic set has dimension $d$ with $0\leq d\leq 6$. However, because the dependence in $(\Re(w_0),\Im(w_0),\omega)^T$ is linear, if it had dimension $d\geq 4$, it would entail the existence of a position $(r_1,r_2,\alpha)^T$ and a non-zero velocity $U_0\in\mathbf R^3$ such that $\Theta_N(r-\nu,\alpha)\mathcal G_NU_0=0$ for all $N\geq 1$. As already mentioned before, this case is only possible if the solid is in the list of Theorem~\ref{theo:stelth}. 

We do not know if there exist solids such that $0<d\leq 3$. Observe that in Section~\ref{sec:contre_exx}, we have only given examples for which the solids can occupy a finite number of different positions, so $d=0$ in these cases. 
\end{itemize}
For all $N\geq 1$, we can invert the system \eqref{main:eq} to obtain:
\begin{align}
\mathcal G_NU&=e^{-i\alpha D_N}e^{\log(r-\nu)D_N}e^{-S_ND_N}e^{-\log(r-\nu)D_N}\Lambda_N(\nu),\\
&=\Theta_N(r-\nu,\alpha)^{-1}\Lambda_N(\nu),
\end{align}
or equivalently, with the notation of equation \eqref{system:1},
\begin{equation}
\label{system:2}
-\mathcal A_n\bar w_0+\mathcal B_nw_0+i\omega\mathcal C_n=e^{-in\alpha}\left[\sum_{k=1}^n\binom{n-1}{k-1}(\nu-r)^{n-k}\lambda_k(\nu)\right],
\end{equation}
for all $n\in\mathbf N$. In this form, we can easily prove:
\begin{prop}
\label{prop:velo}
If the solid does not occur in Theorem~\ref{theo:stelth} and its position is given, then we can deduce its velocity.
\end{prop}
\begin{proof}
Denote $\mathcal G^j$ ($j=1,2,3$)  the complex sequences respectively defined by $\mathcal G^1_k:=(-\mathcal A_k+\mathcal B_k)_{k\geq 1}$, $\mathcal G^2_k:=(i(\mathcal A_k+\mathcal B_k))_{k\geq 1}$ and $\mathcal G^3_k:=(i\mathcal C_k)_{k\geq 1}$. If these sequences  were not $\mathbf R$-linearly independent in $\mathbf C^{\mathbf N}$, it would exist $(U_1,U_2,U_3)^T\neq 0$ in $\mathbf R^3$ such that $\sum_j U_j\mathcal G^j=0$ and then, for any position $(r,\alpha)$ and any $N\geq 1$, we would have $\Theta_N(r-\nu,\alpha)\mathcal G_NU=0$ which contradicts the assumption that the solid is not listed in Theorem~\ref{theo:stelth}. Conversely, if the sequences $\mathcal G^i$ are $\mathbf R$-linearly independent, then there exists $N_0\geq 3$ such that for all $N\geq N_0$, $\mathcal G_N$ is of rank 3 and the proof is completed. 
\end{proof}
\begin{prop}
\label{prop:first:detect}
Assume that the shape $\mathcal S_0$ of the solid, described by the conformal mapping \eqref{general:form}, is such that $e^{i\pi/2}\mathcal S_0=\mathcal S_0$ (the shape of the solid is invariant by rotation of angle $\pi/2$ and center 0). 
Then, from the holomorphic potential $\xi$, we can always deduce the values of $r$, $w_0e^{i\alpha}$ (the linear velocity expressed in a reference fixed frame)  and  $|\omega|$ (the absolute value of the rotational velocity). 
\end{prop}
\begin{proof}
The assumption on the shape $\mathcal S_0$ means that $f(\Omega)=\tilde f(\Omega)$ where $\tilde f$ is the conformal mapping defined by:
$$\tilde f(z):=\tilde c_1z+\sum_{k\leq 1}\tilde c_kz^k,\quad(z\in\Omega),$$
with $\tilde c_k=e^{i\pi/2}c_k$ for any $k=1$ and $k\leq -1$. Replacing the function $f$ by $\tilde f$ in the computations, we obtain that:
\begin{equation}
\label{new:express}
\lambda_n(\nu)=\sum_{k=1}^n\binom{n-1}{k-1}e^{ik\alpha}(r-\nu)^{n-k}\Big[-\mathcal A_{k}e^{i(k+1)\pi/2}\bar w_0+\mathcal B_{k}w_0e^{i(k-1)\pi/2}+i\omega\mathcal C_{k}e^{ik\pi/2}\Big],
\end{equation}
and the coefficients $\lambda_k(\nu)$, defined equivalently by \eqref{system:1} and by \eqref{new:express}, must be equal for all $r,\nu\in\mathbf C$, $\alpha\in\mathbf R/2\pi$ and all $w_0\in\mathbf C$ and $\omega\in\mathbf R$. In particular, for $r=\nu=0$ and $\alpha=0$, we obtain that $\mathcal A_n(1-e^{i(n+1)\pi/2})w_0-\mathcal B_n(1-e^{i(n-1)\pi/2})\bar w_0=0$  and $\mathcal C_n(1-e^{in\pi/2})\omega=0$ for all $w_0\in\mathbf C$, $\omega\in\mathbf R$ and $n\geq 1$. We deduce that $(\mathcal A_n\neq 0)\Leftrightarrow (n\equiv -1\;[4])$, $(\mathcal B_n\neq 0)\Leftrightarrow (n\equiv 1\;[4])$ and $(\mathcal C_n\neq 0)\Leftrightarrow (n\equiv 0\;[4])$. 

Let us consider the problem of detection now we have some extra information about $\mathcal A_n$, $\mathcal B_n$ and $\mathcal C_n$. Equation \eqref{system:2} with $n=2$, gives $\lambda_1(\nu-r)+\lambda_2(\nu)=0$ for all $\nu\in\mathbf C$. Two cases have to be considered:
\begin{itemize}
\item Either $\lambda_1=0$, which means also that $\lambda_2(\nu)=0$ for all $\nu\in\mathbf C$. It entails, according to equation \eqref{system:2} with $n=1$, that $\mathcal B_1w_0=0$. But $\mathcal B_1=|c_1|^2\neq 0$ and hence $w_0=0$. We get $\omega\neq 0$ otherwise we would have $\xi=0$. Now let  $m$ be the smallest index such that $\mathcal C_m\neq 0$. We know that such an index exists, otherwise it would mean that $\mathcal C_k=0$ for all $k\geq 1$ and hence that any rotational motion of the solid generates a complex potential equal to $0$. This is impossible for a solid not listed in Theorem~\ref{theo:stelth}. By induction on $n$ with equation \eqref{system:2}, we must have now $\lambda_n(\nu)=0$ for all $n<m$. We use then equation \eqref{system:2} with $n=m$ to obtain $i\omega\mathcal C_m=e^{-im\alpha}\lambda_m(\nu)$ and next equation \eqref{system:1} with $n=m+1$ to get $\lambda_{m+1}(\nu)=me^{im\alpha}(r-\nu)i\omega\mathcal C_m$. Combining these two identities, we eventually obtain $\lambda_{m+1}(\nu)=m\lambda_m(\nu)(r-\nu)$, hence we deduce the value of $r$.
\item Or $\lambda_1\neq 0$. In this case, we deduce easily the value of $r$ from the relation $\lambda_1(\nu-r)+\lambda_2(\nu)=0$. 
\end{itemize}
Then, equation \eqref{system:2} with $n=1$ gives us the value of $w_0e^{i\alpha}$. There exists at least one $m>0$ such that $\mathcal C_m\neq 0$. Once again, equation \eqref{system:2} with $n=m$ allows us to deduce the value of $|\omega|$.
\end{proof}
If we have more information on $\mathcal S_0$, we can go further in our study:
\begin{prop}
\label{prop:first:detect:2}
In the preceding proposition, assume furthermore that one of the following assumption is satisfied:
\begin{enumerate}
\item $w_0\neq 0$ and there exists an integer $m$ such that either $\mathcal A_m$ and $\mathcal A_{m+4}$ or $\mathcal B_m$ and $\mathcal B_{m+4}$ are both different from 0;
\item\label{second} $\omega\neq 0$ and there exists an integer $m$ such that $\mathcal C_m$ and $\mathcal C_{m+4}$ are both different from $0$;
\end{enumerate}
Then, the solid is detectable.
\end{prop}
\begin{proof}
We know that we can compute $r$ and $w_0e^{i\alpha}$. Assume for instance that the second assumption holds (it would be the same reasoning with the first one). In this case, with equation  \eqref{system:2} specifying $n=m$ and $n=m+4$, we can compute $\omega e^{im\alpha}$ and $\omega e^{i(m+4)\alpha}$. We next deduce $e^{i4\alpha}$ and hence the orientation of the solid (because of the symmetry property, $e^{i4\alpha}$ suffices to provide the orientation). Since $m\equiv 0$ $[4]$ (because $\mathcal C_m\neq 0$) we next deduce the values of $e^{im\alpha}$ and $\omega$. We know that there exists at least one index $p$ such that $\mathcal A_p\neq 0$. We use equation \eqref{system:2} with $n=p$ to deduce the value of $\mathcal A_p \bar w_0e^{ip\alpha}$ or equivalently $\bar{\mathcal A}_p w_0e^{-ip\alpha}$. Since necessarily $p\equiv -1$ $[4]$ then $-p\equiv 1$ $[4]$ and we can deduce the values of both $\omega$ and $e^{i4\alpha}$ with $w_0e^{i\alpha}$.
\end{proof}
To conclude this section, we give examples of solids without any symmetry, which are detectable (one is pictured in Figure~\ref{fig:32}):
\subsection{Example of detection}
\label{ex:detec}
We consider a shape $\mathcal S_0$ described by a complex sequence $c$ such that $c_1\neq 0$, $c_{-4}\neq 0$ and $c_{-7}\neq 0$, all the other coefficients $c_k$ being null. Direct computations lead to:
$$\begin{array}{|c|c|c|c|c|c|c|c|c|}
\hline
k&1&2&3&4&5&6&7&8\\
\hline
\mathcal A_k&0&0&0&c_1^4c_{-4}&0&0&c_1^7c_{-7}&0\\
\hline
\mathcal B_k&|c_1|^2&0&0&0&0&|c_1|^2c_1^4c_{-4}&0&0\\
\hline
\mathcal C_k&0&0&c_1^3\bar c_{-4}c_{-7}&0&2|c_1|^2c_1^4c_{-4}&0&0&(2|c_1|^2+3|c_{-4}|^2)c_1^7c_{-7}\\
\hline
\end{array}$$
Substiting ing these values into System \eqref{system:2}, we obtain that:
\begin{subequations}
\begin{align}
\mathcal B_1 w_0&= e^{-i\alpha}\lambda_1,\label{eq:1}\\
0&= e^{-i2\alpha}[\lambda_1(\nu-r)+\lambda_2(\nu)],\label{eq:2}\\
\mathcal C_3 i\omega &=e^{-i3\alpha}\left[\sum_{k=1}^3\binom{n-1}{k-1}(\nu-r)^{n-k}\lambda_k(\nu)\right],\label{eq:3}\\
-\mathcal A_4\bar w_0 & = e^{-i4\alpha}\left[\sum_{k=1}^4\binom{n-1}{k-1}(\nu-r)^{n-k}\lambda_k(\nu)\right].\label{eq:4}
\end{align}
\end{subequations}
\begin{itemize}
\item If $\omega\neq 0$ and $w_0\neq 0$: From equation \eqref{eq:1} we deduce that $\lambda_1\neq 0$ and from \eqref{eq:2} we deduce the value of $r$. Equation \eqref{eq:3} allows us to determine $|\omega|$ and $3\alpha$ up to $\pi$. Combining next equation \eqref{eq:1} and \eqref{eq:4}, we get $5\alpha$. Using Bezout's identity: $\alpha=u3\alpha+v5\alpha$ with $u=2$ and $v=-1$, we get $\alpha$. We determine $w_0$ with equation \eqref{eq:1} and $\omega$ with equation \eqref{eq:3}. 
\item If $\omega=0$, $w_0\neq 0$: We compute $r$ as in the preceding case. Then, we need further calculations:
\begin{subequations}
\begin{align}
\mathcal B_6 w_0 &= e^{-i6\alpha}\left[\sum_{k=1}^6\binom{n-1}{k-1}(\nu-r)^{n-k}\lambda_k(\nu)\right],\label{eq:11}\\
-\mathcal A_7\bar w_0 & = e^{-i7\alpha}\left[\sum_{k=1}^7\binom{n-1}{k-1}(\nu-r)^{n-k}\lambda_k(\nu)\right].\label{eq:14}
\end{align}
\end{subequations}
With  \eqref{eq:11} and \eqref{eq:1} we get $5\alpha$ and with  \eqref{eq:14} and \eqref{eq:1} we get $8\alpha$. Since $5$ and $8$ are coprime numbers, we deduce the value of $\alpha$. We conclude as in the preceding case.
\item If $\omega\neq 0$ and $w_0=0$: With \eqref{eq:1} and \eqref{eq:2} we deduce that $\lambda_1=\lambda_2(\nu)=0$ for all $\nu\in\mathbf C$. We rewrite \eqref{eq:3} and \eqref{eq:4} as
\begin{align*}
\mathcal C_3i\omega&=e^{-i\alpha}\lambda_3,\\
0&=\lambda_4(\nu)+3(\nu-r)\lambda_3,
\end{align*}
and we deduce first that $\lambda_3\neq 0$ and then the value of $r$.
We next add the equations:
\begin{subequations}
\begin{align}
\mathcal C_5 i\omega&=e^{-i5\alpha}\left[\sum_{k=1}^5\binom{n-1}{k-1}(\nu-r)^{n-k}\lambda_k(\nu)\right],\label{eq:111}\\
\mathcal C_8 i\omega&=e^{-i8\alpha}\left[\sum_{k=1}^5\binom{n-1}{k-1}(\nu-r)^{n-k}\lambda_k(\nu)\right].\label{eq:112}
\end{align}
\end{subequations}
Equations \eqref{eq:3} and \eqref{eq:111} give us $2\alpha$ and \eqref{eq:111} and \eqref{eq:112} give us $3\alpha$. Since $2$ and $5$ are coprime numbers we get $\alpha$ and then $\omega$ with  \eqref{eq:3}.
\end{itemize}
\begin{figure}[h]     
    \centering
     {\includegraphics[height=.3\textwidth]{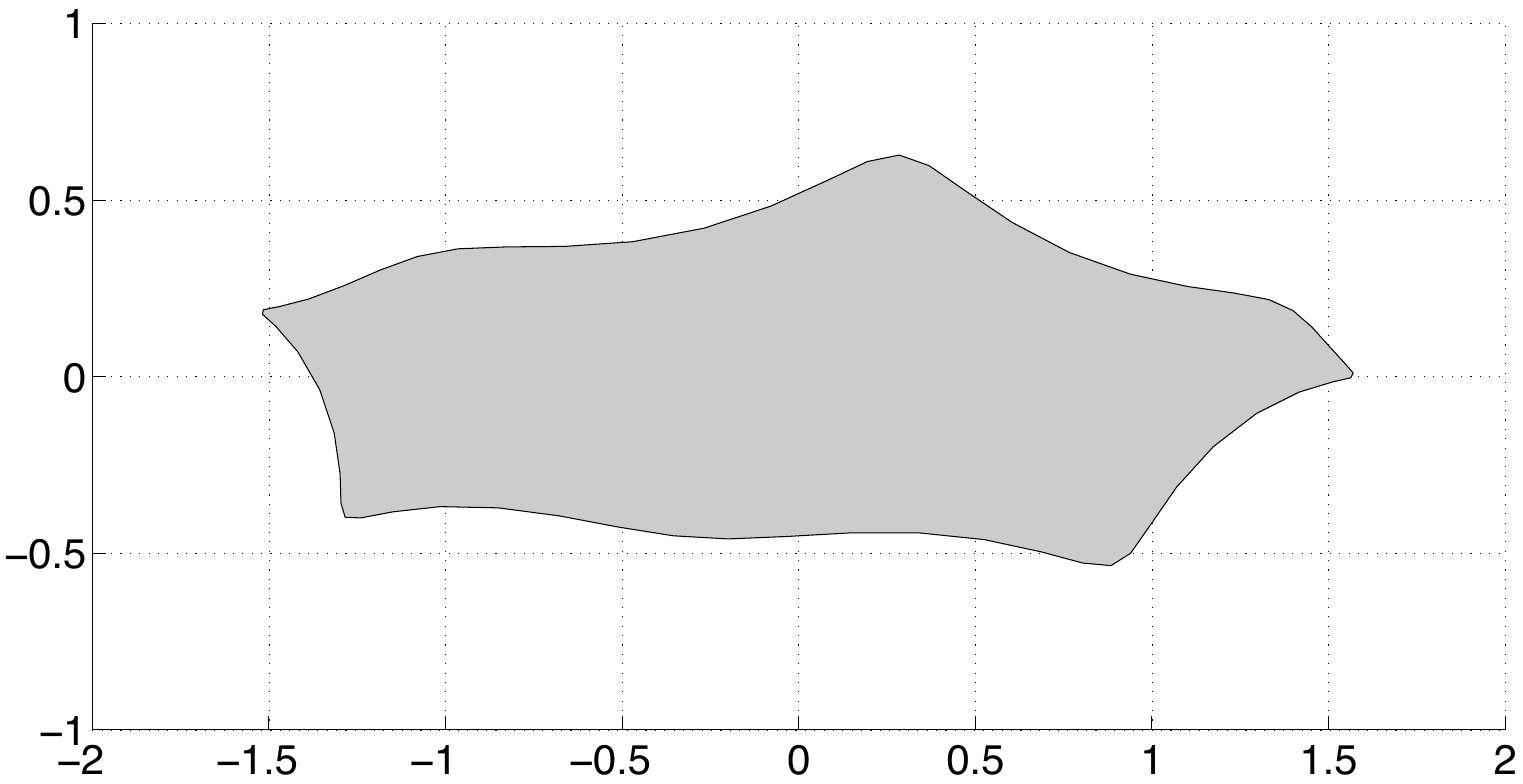}}
     \caption{\label{fig:32}Example of a detectable solid as described in Subsection~\ref{ex:detec}.}
     \end{figure}

\section{Tracking}
\label{tracking}
In this Section, we perform the proof of Theorem~\ref{theo:tracking}. 
So we assume that we know the complex potential for all $t$ in a time interval $[0,T]$ ($T>0$). We have the expression 
\begin{equation}
\label{a:priori:1}
\xi(t,z):=\sum_{j\geq 1}\frac{\lambda_j(t,\nu)}{(z-\nu)^j},\quad |z-\nu|>R(t,\nu),
\end{equation}
where $\lambda_j(t,\nu)$ are complex numbers and $R(t,\nu):=\limsup_{j\to+\infty}|\lambda_j(t,\nu)|^{1/j}$. At any time $t$, the series is uniformly convergent on $\{z\in\mathbf C\,:\, |z-\nu|>R(t,\nu)\}$. For all $N\geq 1$, we denote $\Lambda_N(t,\nu):=(\lambda_1(t,\nu),\ldots,\lambda_N(t,\nu))^T$ and we have, according to the results of the preceding section:
$$\mathcal G_NU(t)=\Theta_N(r(t)-\nu,\alpha(t))^{-1}\Lambda_N(t,\nu),$$
where we recall that $U(t):=(\Re(w_0(t)),\Im(w_0(t)),\omega(t))^T$. In the proof of Proposition~\ref{prop:velo} we have shown that if the solid is not one of those described in Theorem~\ref{theo:stelth}, then there exists $N\geq 1$ such that $\mathcal G_N$ has rank 3. It means that there exists (at least) one inverse $\mathcal G_N^{-1}$ allowing one to express the velocity as:
$$U(t):=\mathcal G_N^{-1}\Theta_N(r(t)-\nu,\alpha(t))^{-1}\Lambda_N(t,\nu).$$
We next get:
$$\begin{pmatrix}
e^{i\alpha(t)}w_0(t)\\
\omega
\end{pmatrix}
=\begin{pmatrix}
e^{i\alpha(t)}&ie^{i\alpha(t)}&0\\
0&0&1\end{pmatrix}\mathcal G_N^{-1}\Theta_N(r(t)-\nu,\alpha(t))^{-1}\Lambda_N(t,\nu).$$
This equation can be rewritten as:
$$\frac{d}{dt}\begin{pmatrix}
r(t)\\
\alpha(t)\end{pmatrix}=\begin{pmatrix}
e^{i\alpha(t)}&ie^{i\alpha(t)}&0\\
0&0&1\end{pmatrix}\mathcal G_N^{-1}\Theta_N(r(t)-\nu,\alpha(t))^{-1}\Lambda_N(t,\nu),$$
to which we can apply the Cauchy-Lipschitz Theorem. The proof is then completed.
\section{Conclusion}
\label{open}
In this article, we have proved that not all solids moving in a perfect fluid can be detected by measuring the potential of the fluid. This observation has led us to define the notion of {\it detectable} solids, which is a purely geometric property. When the geometry is described by means of a conformal mapping, we were able to exhibit examples of detectable (or partially detectable) solids. However, the complete characterization of such solids in terms of the complex sequence $(c_k)_{k\in\mathbf Z}$ remains to be done.
\bibliographystyle{abbrv}
\bibliography{detection1}

\begin{thebibliography}{1}

\bibitem{Alvarez:2005aa}
C.~Alvarez, C.~Conca, L.~Friz, O.~Kavian, and J.~H. Ortega.
\newblock Identification of immersed obstacles via boundary measurements.
\newblock {\em Inverse Problems}, 21(5):1531--1552, 2005.

\bibitem{Call:1993aa}
G.~S. Call and D.~J. Velleman.
\newblock Pascal's matrices.
\newblock {\em Amer. Math. Monthly}, 100(4):372--376, 1993.

\bibitem{Conca:2008ab}
C.~Conca, P.~Cumsille, J.~Ortega, and L.~Rosier.
\newblock On the detection of a moving obstacle in an ideal fluid by a boundary
  measurement.
\newblock {\em Inverse Problems}, 24(4):045001, 18, 2008.

\bibitem{Doubova:2007aa}
A.~Doubova, E.~Fern{\'a}ndez-Cara, and J.~H. Ortega.
\newblock On the identification of a single body immersed in a
  {N}avier-{S}tokes fluid.
\newblock {\em European J. Appl. Math.}, 18(1):57--80, 2007.

\bibitem{Heck:2007aa}
H.~Heck, G.~Uhlmann, and J.-N. Wang.
\newblock Reconstruction of obstacles immersed in an incompressible fluid.
\newblock {\em Inverse Probl. Imaging}, 1(1):63--76, 2007.

\bibitem{Milne-Thomson:1960aa}
L.~M. Milne-Thomson.
\newblock {\em Theoretical hydrodynamics}.
\newblock 4th ed. The Macmillan Co., New York, 1960.

\bibitem{Rudin:1987aa}
W.~Rudin.
\newblock {\em Real and complex analysis}.
\newblock McGraw-Hill Book Co., New York, 3rd edition, 1987.

\end{thebibliography}
\end{document}